\pdfoutput=1
\documentclass[english]{jnsao}
\usepackage[utf8]{inputenc}
\usepackage{csquotes}
\usepackage{babel}
\usepackage{amsmath,amsthm}
\usepackage[nameinlink,capitalise]{cleveref}
\usepackage{tikz}
\usepackage{relsize}

\theoremstyle{definition}
\newtheorem{assumption}[theorem]{Assumption}

  \manuscriptsubmitted{2021-02-24}
  \manuscriptaccepted{2021-08-05}
  \manuscriptvolume{2}
  \manuscriptnumber{7215}
  \manuscriptyear{2021}
  \manuscriptdoi{10.46298/jnsao-2021-7215}

\newcommand\norm[2]{\left\Vert#1\right\Vert_{#2}}

\newcommand\N{\mathbb{N}}
\newcommand\R{\mathbb{R}}

\newcommand\B{\mathbb{B}}

\newcommand\tto{\rightrightarrows}

\newcommand\id{\operatorname{id}}

\newcommand{\spa}{\operatorname{span}}

\newcommand{\dist}{\operatorname{dist}}

\newcommand{\dom}{\operatorname{dom}}
\newcommand{\rge}{\operatorname{rge}}
\newcommand{\gph}{\operatorname{gph}}
\newcommand{\epi}{\operatorname{epi}}

\DeclareMathOperator*{\argmin}{\operatorname{argmin}}

\newcommand{\mv}{\,\mid\,}

\crefname{figure}{Figure}{Figures}

\allowdisplaybreaks

\numberwithin{equation}{section}

\usetikzlibrary{arrows}

\tikzset{
    punkt/.style={
           rectangle,
           draw=white, very thick,
           text width=10em,
           minimum height=1.5em,
           text centered}
}

\title{On implicit variables in optimization theory}

\author{%
	Mat\'{u}\v{s} Benko
	\thanks{
	Applied Mathematics and Optimization.
	University of Vienna, Austria.
	\email{matus.benko@univie.ac.at}.
	ORCID: 0000-0003-3307-7939.
	}
	\and
	Patrick Mehlitz
	\thanks{
	Institute of Mathematics, Chair of Optimal Control.
	Brandenburgische Technische Universität Cottbus--Senftenberg, Germany.
	\email{mehlitz@b-tu.de}.
	ORCID: 0000-0002-9355-850X.
	}
}
\shortauthor{Benko, Mehlitz}

\shorttitle{On implicit variables in optimization theory}

\acknowledgements{
    The research of the first author was supported by the Austrian Science Fund
	(FWF) under grants P29190-N32 and P32832-N.
	Moreover, the first author would like to thank the
	Institute of Computational Mathematics, Johannes Kepler University Linz,
	for access to the infrastructure of the institute. 
}

\begin{document}

\maketitle

\begin{abstract}
    Implicit variables of a mathematical program are variables which
	 do not need to be optimized but are used to model feasibility
	 conditions.
	 They frequently appear in several different problem classes
	 of optimization theory comprising bilevel programming, evaluated 
	 multiobjective optimization, or nonlinear optimization problems
	 with slack variables.
	 In order to deal with implicit variables, they are often
	 interpreted as explicit ones. 
	 Here, we first point out that this is a light-headed approach
	 which induces artificial locally optimal solutions.
	 Afterwards, we derive various
	 Mordukhovich-stationarity-type necessary optimality conditions which correspond to
	 treating the implicit variables as
	 explicit ones on the one hand,
	 or using them only implicitly to model
	 the constraints on the other.
	 A detailed comparison of the obtained stationarity conditions as well as the associated underlying
	 constraint qualifications will be provided.
	 Overall, we proceed in a fairly general setting
	 relying on modern tools of variational analysis.
	 Finally, we apply our findings to different well-known problem classes
	 of mathematical optimization in order to visualize the obtained
	 theory. 
	 \\[2ex]
	\noindent
	\emph{Keywords:}
		 	Implicit variables, Local minimizers, Metric subregularity, M-stationarity, 
			Variational analysis
	\\[2ex]
	\noindent
	\emph{MSC (2010):} 
		49J53, 90C30, 90C33
\end{abstract}

\section{Introduction}\label{sec:introduction}

In this paper, we consider the mathematical program
\begin{equation}\label{eq:implicit_multipliers}\tag{P}
	\begin{split}
		f(z)&\,\to\,\min\limits_z\\
		0&\,\in\,H(z)\\
		z&\,\in\,M
	\end{split}
\end{equation}
where $f\colon\R^n\to\R$ is a given locally Lipschitz continuous 
objective function
and the set-valued mapping $H\colon\R^n\tto\R^s$ attains the form
\begin{equation*}
 H(z) := \bigcup\limits_{\lambda\in F(z)}G(z,\lambda)
\end{equation*}
for set-valued mappings $F\colon\R^n\tto\R^m$ and 
$G\colon\R^n\times\R^m\tto\R^s$ which possess
a closed graph, respectively.
Furthermore, $M\subset\R^n$ is a nonempty and closed set of simple
structure, e.g., a typical constraint set defined via standard
inequality and equality constraints satisfying a suitable constraint qualification.
The feasible set of \eqref{eq:implicit_multipliers} will be denoted by $Z\subset\R^n$.
We emphasize that the objective in \eqref{eq:implicit_multipliers} is only
minimized w.r.t.\ $z$ and \emph{not} w.r.t.\ $\lambda$
although the latter variable is used implicitly in order to model $Z$.
That is why we call $\lambda$ an \emph{implicit} variable.
Let us note that the rather general model \eqref{eq:implicit_multipliers}
covers several problem classes from optimization theory, which are 
modeled with the aid of implicit variables, e.g., bilevel programming problems, 
evaluated multiobjective 
programs, or programs with slack variables, see \Cref{sec:examples} as well.
In these examples, the implicit variables are given by means of
the lower level Lagrange multipliers, scalarization parameters, or
the introduced slack variables, respectively.

The main difficulty in \eqref{eq:implicit_multipliers}
is the appearing union of image sets associated with the given
set-valued mappings $F$ and $G$. An easy approach which can be used 
to handle this issue is to interpret $\lambda$ as a
meaningful variable. This leads to the consideration of
\begin{equation}\label{eq:explicit_multipliers}\tag{Q}
	\begin{split}
		f(z)&\,\to\,\min\limits_{z,\lambda}\\
			\lambda&\,\in\,F(z)\\
			0&\,\in\,G(z,\lambda)\\
			z&\,\in\,M
	\end{split}
\end{equation}
where $\lambda$ now takes the role of an \emph{explicit} variable.
We denote the feasible set of this program by $\tilde Z\subset\R^n\times\R^m$.
As we will see in \Cref{sec:relationship_of_P_and_Q}, there
is a one-to-one correspondence between the global minimizers
of \eqref{eq:implicit_multipliers} and \eqref{eq:explicit_multipliers}.
However, the relationship between local minimizers is far more delicate.
Particularly, there may exist local minimizers of \eqref{eq:explicit_multipliers}
which do not correspond to local minimizers of \eqref{eq:implicit_multipliers}.
Similar issues are pointed out in the context of bilevel programming,
see
\cite{AusselSvensson2019,DempeDutta2012,DempeMefoMehlitz2018,DempeMehlitz2018,LamparielloSagratellaStein2019}, 
w.r.t.\ the use of slack variables in logical programming, 
see \cite{Mehlitz2019b,Mehlitz2019}, 
or cardinality-constrained
optimization, see \cite{BurdakovKanzowSchwartz2016}, and will be
generalized in this paper.

Let us briefly mention that the concept of implicit variables can be extended
to related areas of optimization theory like the numerical solution of generalized
equations resulting from optimality conditions of convex optimization programs.
Here, it is also common to introduce intermediate or slack variables in order to 
state the convex chain rule for compositions, see e.g.\ \cite[Section~4]{BotCsetnekNagy2013},
or to decompose the inclusion to make it more tractable, see \cite[Section~5]{GfrererOutrata2019}.
It sometimes may be beneficial, however,
to formulate the assumptions ensuring convergence
in the implicit-variable-free setting.

Observing that the computation of local minimizers of \eqref{eq:explicit_multipliers}
does not generally yield local minimizers of \eqref{eq:implicit_multipliers} while
noting that \eqref{eq:explicit_multipliers} is likely to be a nonconvex program
due to the underlying applications we have in mind, a direct treatment
of \eqref{eq:implicit_multipliers} seems to be reasonable. 
With the aid of modern variational analysis,
in particular limiting coderivatives of set-valued
mappings,
see \cite{Mordukhovich2018,RockafellarWets1998}, we will infer necessary optimality
conditions of Mordukhovich-stationarity-type (subsequently, we will use the term
M-stationarity for brevity) for \eqref{eq:implicit_multipliers} under mild
constraint qualifications.
More precisely, we first derive optimality conditions comprising the implicitly 
known coderivative of the mapping $H$
under metric subregularity of $H$.
Afterwards, we study assumptions which
allow us to state these optimality conditions 
in terms of the coderivatives associated with $F$ and $G$. Clearly, this makes the
utilization of a coderivative \emph{chain rule} necessary, 
so we discuss weak conditions ensuring its applicability. Therefore, we make use of
the recent results from \cite{Benko2019,BenkoMehlitz2020} which are mainly based on
the presence of \emph{metric subregularity} for so-called ``feasibility'' mappings
and \emph{inner semicompactness} for so-called ``intermediate'' mappings.
Interestingly, these assumptions are mainly inherent for the example problems
from \cref{sec:examples} which underlines the power of this approach.
On the other hand, using problem \eqref{eq:explicit_multipliers}
as the starting point enables us to avoid the (coderivative of)
the mapping $H$. A similar role is then played by an auxiliary mapping, 
which depends on both, $z$ and $\lambda$.
As it will turn out, the procedures sketched above lead to three stationarity systems and 
diverse constraint qualifications which differ from each other w.r.t.\ their \emph{degree of
explicitness}.
We will provide a detailed comparison of all these stationarity conditions and constraint
qualifications, and we will comment on their respective relation to the problem \eqref{eq:explicit_multipliers}.
In \cite{AdamHenrionOutrata2018}, the authors discuss related issues
by means of an equilibrium-constrained mathematical problem with lower level inequality
constraints. It turns out that there is a significant difference
between the M-stationarity conditions for problem \eqref{eq:implicit_multipliers}, which are formulated
without taking into account the special structure of $H$
and, hence, without the variable $\lambda$,
and the M-stationarity conditions of \eqref{eq:explicit_multipliers}.
Moreover, the difference between the corresponding constraint qualifications,
i.e., the conditions that guarantee the validity of the two types
of M-stationarity at local minimizers, is even more striking.
Here, we generalize and deepen the approach from \cite{AdamHenrionOutrata2018}.
Let us briefly note that the theory of this paper can be extended
directly to the situation where $\R^n$, $\R^m$, and $\R^s$ are 
replaced by arbitrary Euclidean spaces. For simplicity, however,
we focus on the most elementary setting.

The remaining parts of this paper are organized as follows:
In \cref{sec:examples}, we present three prototypical example classes of optimization
problems where implicit variables are used to model the feasible set,
namely bilevel optimization problems, evaluated multiobjective optimization
problems, and mathematical programs with cardinality constraints.
We provide an overview of the 
notation used in this manuscript before recalling the definitions and some
background information about the fundamental tools of variational analysis
we are going to exploit here in \cref{sec:notation_preliminaries}. 
Furthermore, we motivate a rather abstract notion
of M-stationarity for mathematical problems with generalized equation
constraints and provide some calculus rules for the coderivative of
compositions and products of set-valued mappings.
Afterwards, 
\Cref{sec:lambda_as_variable} is dedicated to the abstract analysis of
the program \eqref{eq:implicit_multipliers} via \eqref{eq:explicit_multipliers}. 
First, we study the relationship
between the solutions of \eqref{eq:implicit_multipliers} and \eqref{eq:explicit_multipliers}
in \cref{sec:relationship_of_P_and_Q}. Second, we derive three potentially
distinct notions of M-stationarity for \eqref{eq:implicit_multipliers} as well
as associated constraint qualifications in \cref{sec:M_Stationarity_conditions}. 
Furthermore, we clarify how to obtain
explicit M-stationarity conditions in terms of initial problem data from the abstract
M-stationarity conditions of \eqref{eq:implicit_multipliers}.
As we will demonstrate in \cref{sec:optimality_conditions_via_special_structure}, the
necessary assumptions are inherently satisfied in many practically relevant settings.
Finally, we briefly comment on the sufficiency of the introduced M-stationarity
notions in the presence of convexity in \cref{sec:sufficient_conditions}.
In \cref{sec:applications}, we apply some of our findings to the aforementioned example
problems. As it will turn out, we recover or even enhance available results from
the literature with our approach.
Finally, we present some concluding remarks in \cref{sec:conclusions}.

\section{Motivating examples}\label{sec:examples}

In this section, we present three prototypical classes of
optimization problems whose respective feasible sets can
be modeled with the aid of implicit variables.

\subsection{Bilevel programming}\label{sec:bilevel_programming}

Quite often, hierarchical decision making appears naturally in real world problems raising in
economics, logistics, natural sciences, or engineering. In case where two decision makers
are involved, so-called bilevel optimization problems can be used to model and study the
underlying applications theoretically and numerically, see 
\cite{Bard1998,Dempe2002,DempeKalashnikovPerezValdesKalashnykova2015,Mordukhovich2018,ShimizuIshizukaBard1997}. 
For given parameters $x\in\R^{n_1}$, let us consider the parametric optimization problem
\begin{equation}\label{eq:lower_level}\tag{LL}
	\begin{split}
		j(x,y)&\,\to\,\min\limits_y\\
		g(y)&\,\in\,C
	\end{split}
\end{equation}
where $j\colon\R^{n_1}\times\R^{n_2}\to\R$ is twice continuously differentiable and convex
w.r.t.\ $y$ for each $x\in\R^{n_1}$ while $g\colon\R^{n_2}\to\R^m$ is twice continuously
differentiable and convex w.r.t.\ the closed, convex cone $C\subset\R^m$.
Let us set $\Gamma:=\{y\in\R^{n_2}\,|\,g(y)\in C\}$.
One generally refers to \eqref{eq:lower_level} as
the \emph{lower level} or \emph{follower's problem}. 
It is well known that due to the convexity of $j(x,\cdot)$ and $\Gamma$, the set of
optimal solutions associated with \eqref{eq:lower_level} is equivalently
characterized by
\[
	-\nabla_yj(x,y)\in\widehat{\mathcal N}_\Gamma(y).
\]
Under a suitable constraint qualification,
we can characterize the set of optimal solutions associated with \eqref{eq:lower_level} 
equivalently by the generalized equation
\[
	-\nabla_yj(x,y)\in\nabla g(y)^\top\widehat{\mathcal N}_{C}(g(y)).
\]
Above, we exploited the notion of the regular normal cone which coincides with the
normal cone of convex analysis due to convexity of $\Gamma$ and $C$, see
\cref{sec:notation_preliminaries}.

The superordinate so-called \emph{upper level} or \emph{leader's problem}, given by
\[	
	\begin{split}
	f(x,y)&\,\to\,\min\limits_{x,y}\\
		x&\,\in\,S\\
		y&\,\in\,\Psi(x),
	\end{split}
\]
where $S\subset\R^{n_1}$ is a closed set and $\Psi\colon\R^{n_1}\tto\R^{n_2}$ is the solution mapping associated with \eqref{eq:lower_level}, is equivalent to
\begin{equation}\label{eq:BPP_implicit_multipliers}\tag{BPP}
	\begin{split}
		f(x,y)&\,\to\min\limits_{x,y}\\
		x&\,\in\,S\\
		0&\,\in\,\bigcup_{\lambda\in\widehat{\mathcal N}_{C}(g(y))}
							\left\{\nabla_yj(x,y)+\nabla g(y)^\top\lambda\right\}
	\end{split}
\end{equation}
whenever the aforementioned assumptions are valid.
The latter is a program of type \eqref{eq:implicit_multipliers} in the sense
\begin{equation}\label{eq:setting_bilevel_programming}
		F(z):=\widehat{\mathcal N}_C(g(y)),\quad
		G(z,\lambda):=\nabla_yj(x,y)+\nabla g(y)^\top\lambda,\quad
		M:=S\times\R^{n_2}
\end{equation}
where we used $z:=(x,y)$ and $n:=n_1+n_2$.
Particularly, the implicit variable $\lambda$ is the Lagrange multiplier associated with the 
constraints in \eqref{eq:lower_level}. It is well known that interpreting $\lambda$ as a variable
in \eqref{eq:BPP_implicit_multipliers} may cause difficulties w.r.t.\ local minimizers, 
see e.g.\ 
\cite{AusselSvensson2019,DempeDutta2012,DempeMefoMehlitz2018,LamparielloSagratellaStein2019}.
The authors in \cite{AdamHenrionOutrata2018},
where the particular setting $C:=\R^m_-$ is discussed, focused on a qualitative comparison
of the reformulation
\[
	\begin{split}
		f(x,y)&\,\to\min\limits_{x,y}\\
		x&\,\in\,S\\
		-\nabla_yj(x,y)&\,\in\,\widehat{\mathcal N}_\Gamma(y),
	\end{split}
\]
which is fully equivalent to the original bilevel programming problem, and the
problem \eqref{eq:explicit_multipliers} associated with \eqref{eq:BPP_implicit_multipliers}
regarding constraint qualifications and M-stationarity-type optimality conditions.
They came up with the observation that the respective stationarity systems as well
as constraint qualifications differ significantly.
The analysis in this paper will put their results into a more general context.

\subsection{Evaluating weakly efficient points in multicriteria optimization}\label{sec:vector_optimization}

For a twice continuously differentiable componentwise convex function $j\colon\R^n\to\R^m$ with $m\geq 2$ components and a nonempty, closed, convex set $\Gamma\subset \R^n$, 
we consider the following multicriteria optimization problem which is a standard model
from vector optimization, see \cite{Ehrgott2005,Jahn2004}:
\begin{equation}\label{eq:MOP}\tag{MOP}
	\begin{split}
		j(z)&\,\rightarrow\,\text{``}\min\text{''}\\
		z&\,\in\,\Gamma.
	\end{split}
\end{equation}
The quotation marks emphasize that classical minimization is not
possible since the componentwise natural ordering in $\R^m$ only
provides a partial order which is not total. 
In order to overcome this problem, several different notions of
so-called \emph{efficiency} have been derived in order to
characterize reasonable feasible points of \eqref{eq:MOP}.
In this context, a point $\bar z\in \Gamma$ is called \emph{weakly efficient} for \eqref{eq:MOP}
if the condition 
\[
	\bigl(j(\bar z)-\R^m_{++}\bigr)\cap j(\Gamma)=\varnothing
\]
holds where we used $\R^m_{++}:=\{y\in\R^m\,|\,\forall i\in\{1,\ldots,m\}\colon\,y_i>0\}$.
Let $\Gamma_\textup{we}\subset\R^n$ be the set of weakly efficient points
associated with \eqref{eq:MOP}. Using a classical linear scalarization approach, see \cite{Ehrgott2005},
one obtains the characterization
\[
	\bar z\in\Gamma_\textup{we}
	\quad\Longleftrightarrow\quad
	\exists\lambda\in \Delta\,\forall z\in\Gamma\colon\,
	\lambda^\top j(\bar z)\leq\lambda^\top j(z)
\]
 of weak efficiency since the problem data in \eqref{eq:MOP} is convex.
Above, $\Delta\subset\R^m$ given by
\[
	\Delta:=\bigl\{
				\lambda\in\R^m\,|\,
				\lambda\geq 0,\,
				\mathsmaller\sum\nolimits_{i=1}^m\lambda_i=1
			\bigr\}
\]
denotes the standard simplex.
In practice, the set $\Gamma_{\textup{we}}$ will be quite large in general
which is why it is desirable to shrink it for applicability purposes using
another decision criterion given in form of a scalar function, see 
e.g.\ \cite{Benson1986,Bolintineanu1993b,Bolintineanu1993,HorstThoai1999}.
Here, we investigate the
\emph{evaluated multiobjective optimization problem}
\[
	\begin{split}
		f(z)&\,\to\,\min\\
		z&\,\in\,\Gamma_{\textup{we}}.
	\end{split}
\]
Invoking the above arguments, this is equivalent to
\begin{equation}\label{eq:EMOP_implicit_multipliers}\tag{EMOP}
	\begin{split}
		f(z)&\,\to\,\min\\
		0&\,\in\,\bigcup\limits_{\lambda\in \Delta}\Psi(\lambda)-z
	\end{split}
\end{equation}
where $\Psi\colon\R^m\tto\R^n$ is the set-valued mapping defined by
\[	
	\Psi(\lambda):=
	\begin{cases}
		\argmin\limits_z\{\lambda^\top j(z)\,|\,z\in\Gamma\}	&	\lambda\in\Delta,\\
		\varnothing												&	\lambda\notin\Delta.
	\end{cases}
\]
Obviously, this is a program of type \eqref{eq:implicit_multipliers} where the scalarization
parameter takes the role of the implicit variable.
Particularly, the data from \eqref{eq:implicit_multipliers} takes the form
\begin{equation}\label{eq:setting_multiobjective_optimization}
	F(z):= \Delta,\qquad
	G(z,\lambda):=\Psi(\lambda)-z,\qquad
	M:=\R^n.
\end{equation}
It has been reported in \cite{DempeMehlitz2018} that interpreting scalarization parameters 
as variables is generally delicate in multiobjective bilevel programming, and this
observation clearly extends to \eqref{eq:EMOP_implicit_multipliers} as well.

\subsection{Cardinality-constrained programming}\label{sec:CCMP}

For some vector $z\in\R^n$, let $\norm{z}{0}$ denote the number of non-zero entries of
$z$. In order to guarantee \emph{sparsity} of solutions to optimization problems, 
one may consider the so-called \emph{cardinality-constrained mathematical program}
\begin{equation}\label{eq:CCMP}\tag{CCMP}
	\begin{split}
		f(z)&\,\to\,\min\\
		z&\,\in\,M\\
		\norm{z}{0}&\,\leq\,\kappa
	\end{split}
\end{equation}
where $\kappa\in\N$ satisfies $1\leq\kappa\leq n-1$, see e.g.\
\cite{Bienstock1996,BucherSchwartz2018,BurdakovKanzowSchwartz2016,CervinkaKanzowSchwartz2016,
Mehlitz2019a,PanXiuFan2017,PaXiuZhoun2015}
for an overview of existing reformulations, optimality conditions, solution algorithms,
and further references.

Let us rewrite the cardinality constraint $\norm{z}{0}\leq\kappa$ 
as a constraint of type $q(z)\in D$ where $q\colon\R^n\to\R^\ell$ is continuously differentiable
and $D\subset\R^\ell$ is the union of finitely many convex polyhedral sets,
allowing us to interpret \eqref{eq:CCMP} as a so-called \emph{disjunctive} optimization problem,
see e.g.\ \cite{BenkoCervinkaHoheisel2019,FlegelKanzowOutrata2007,Mehlitz2019a}.
Therefore, we introduce $\mathcal J_\kappa:=\{\alpha\in\{0,1\}^n\,|\,\sum_{i=1}^n\alpha_i=\kappa\}$.
Furthermore, for each $\alpha\in\{0,1\}^n$, we set 
\[
	\R^n_\alpha:=\spa\{\mathtt e_i\,|\,i=1,\ldots,n,\,\alpha_i=1\}
\]
where $\mathtt e_i\in\R^n$ denotes the $i$-th unit vector in $\R^n$. Finally, we set
$D_\kappa:=\bigcup_{\alpha\in\mathcal J_\kappa}\R^n_\alpha$.
Then we easily see
\begin{equation}\label{eq:disjunctive_reformulation_cardinality_constraint}
	\forall z\in\R^n\colon\quad
	\norm{z}{0}\leq\kappa\,\Longleftrightarrow\,z\in D_\kappa.
\end{equation}
The variational geometry of $D_\kappa$ has been explored in \cite{PanXiuFan2017,PaXiuZhoun2015}.
Particularly, exploiting the limiting and regular normal cone to $D_\kappa$, reasonable notions of M- and strong stationarity are available for \eqref{eq:CCMP}.

It has been observed in \cite[Theorem~3.2]{BurdakovKanzowSchwartz2016} that we have
\[
	\forall z\in\R^n\colon\quad
	\norm{z}{0}\leq\kappa\,\Longleftrightarrow\,
	\exists\lambda\in\R^n\colon\,
	\mathtt e^\top \lambda\geq n-\kappa,\,(z_i,\lambda_i)\in\mathcal C\;\forall i\in\{1,\ldots,n\}
\]
where we used
\[
	\mathcal C:=\{(a,b)\in\R^2\,|\,ab=0,\,b\in[0,1]\},
\]
and that gave rise to the reformulation of \eqref{eq:CCMP} as
\begin{equation}\label{eq:CCMP_implicit_multipliers}
	\begin{aligned}
	f(z)&\,\to\,\min\limits_{z,\lambda}&&\\
	z&\,\in\,M&&\\
	\mathtt e^\top \lambda&\,\geq\,n-\kappa&\quad&\\
	(z_i,\lambda_i)&\,\in\,\mathcal C&& i=1,\ldots,n.	
	\end{aligned}
\end{equation}
Above, $\mathtt e\in\R^n$ denotes the all-ones vector.
Clearly, the variable $\lambda$ plays the role of an explicit variable in \eqref{eq:CCMP_implicit_multipliers} although it is not relevant for the
purpose of minimization.
In order to state \eqref{eq:CCMP_implicit_multipliers} in the form \eqref{eq:implicit_multipliers},
one may choose
\begin{equation}\label{eq:setting_CCMP}
	F(z)\equiv\bigl\{\lambda\in\R^n\,\big|\,\mathtt e^\top\lambda\geq n-\kappa\bigr\},\qquad
	G(z,\lambda):=\prod_{i=1}^n(\mathcal C-(z_i,\lambda_i)),
\end{equation}
i.e., $s=2n$ holds in this particular case.

It has been reported in \cite{BurdakovKanzowSchwartz2016} that 
\eqref{eq:CCMP} and \eqref{eq:CCMP_implicit_multipliers}
are equivalent w.r.t.\ global minimizers, that the local minimizers of \eqref{eq:CCMP} can be found among
the local minimizers of \eqref{eq:CCMP_implicit_multipliers}, 
and that for a local minimizer $(\bar z,\bar \lambda)\in\R^n\times\R^n$
of \eqref{eq:CCMP_implicit_multipliers}, which satisfies $\norm{\bar z}{0}=\kappa$, 
$\bar z$ is a local minimizer of
\eqref{eq:CCMP}. The situation where $\norm{\bar z}{0}<\kappa$ holds, 
however, has been shown to be crucial since, in this
case, $\bar z$ does not need to be locally optimal for \eqref{eq:CCMP}. 
First- and second-order optimality conditions for
\eqref{eq:CCMP} via its surrogate \eqref{eq:CCMP_implicit_multipliers} have been derived in
\cite{BucherSchwartz2018,BurdakovKanzowSchwartz2016,CervinkaKanzowSchwartz2016} while the authors in
\cite{Mehlitz2019a,PanXiuFan2017,PaXiuZhoun2015} exploited \eqref{eq:disjunctive_reformulation_cardinality_constraint}
in order to infer optimality conditions for \eqref{eq:CCMP} without relying on implicit variables.
To the best of our knowledge, a detailed comparison of both approaches does not exist in the literature.
Some regarding remarks, however, can be found in the papers \cite{BucherSchwartz2018,Mehlitz2019a}.
We also refer the interested reader to \cite{BeckHallak2016} where yet another approach to
cardinality-constrained optimization is discussed which also avoids the appearance of implicit
variables.

\section{Notation and preliminaries}\label{sec:notation_preliminaries}

\subsection{Notation}

Throughout the manuscript, we make use of the standard concepts of variational
analysis, see \cite{Mordukhovich2018,RockafellarWets1998}.

\subsubsection*{Basic notation}

In this manuscript, we equip $\R^n$ with the common Euclidean inner product and
the common Euclidean norm $\norm{\cdot}{}$. In order to extend the notion of
norms to product spaces, we use the sum of the underlying
Euclidean norms to induce a norm in the product space.
Due to notational purposes, for arbitrary $z\in\R^n$ and $w\in\R^m$, we use both notations
\[
	\begin{pmatrix}z\\w\end{pmatrix}
	\qquad\text{and}\qquad
	(z,w)
\]
in order to represent elements of $\R^n\times\R^m$.
We use $\B$ in order to denote the closed unit ball around $0$.
For a set $A\subset\R^n$ and $z\in\R^n$, we exploit $z+A=A+z:=\{z+a\,|\,a\in A\}$
for brevity. Furthermore, we use
\[
	A^\circ:=\bigl\{\xi\in\R^n\,\bigl|\,\forall a\in A\colon\,\xi^\top a\leq 0\},
	\qquad
	A^\perp:=\bigl\{\xi\in\R^n\,\bigl|\,\forall a\in A\colon\,\xi^\top a=0\}
\]
in order to represent the polar cone and the annihilator of $A$. We make use of
$z^\perp:=\{z\}^\perp$. We set
\[	
	\dist(z,A):=\inf\{\norm{a-z}{}\,|\,a\in A\}
\]
in order to denote the distance of $z$ to $A$.
Throughout the manuscript, we exploit $\R_+:=\{s\in\R\,|\,s\geq 0\}$ and
$\R_-:=\{s\in\R\,|\,s\leq 0\}$.

\subsubsection*{Set-valued mappings}

For a set-valued mapping $\Upsilon\colon\R^n\tto\R^m$, the sets
$\gph\Upsilon:=\{(z,w)\in\R^n\times\R^m\,|\,w\in\Upsilon(z)\}$,
$\dom\Upsilon:=\{z\in\R^n\,|\,\Upsilon(z)\neq\varnothing\}$,
$\rge\Upsilon:=\{w\in\R^m\,|\,\exists z\in\R^n\colon\,w\in\Upsilon(z)\}$,
and $\ker\Upsilon:=\{z\in\R^n\,|\,0\in\Upsilon(z)\}$
are called graph, domain, range, and kernel of $\Upsilon$, respectively.
The set-valued mapping $\Upsilon^{-1}\colon\R^m\tto\R^n$, given by
$\Upsilon^{-1}(w):=\{z\in\R^n\,|\,w\in\Upsilon(z)\}$ for all $w\in\R^m$,
is referred to as the inverse of $\Upsilon$.

The mapping $\Upsilon$ is called locally bounded at $\bar z\in\dom\Upsilon$ whenever we find 
a neighborhood $U\subset\R^n$ of $\bar z$ and a bounded set $B\subset\R^m$ such that
$\Upsilon(z)\subset B$ holds for all $z\in U$.
Furthermore, $\Upsilon$ is said to be inner semicompact at $\bar z$ 
w.r.t.\ $\Omega\subset\R^n$ if for each sequence $\{z_k\}_{k\in\N}\subset\Omega$
such that $z_k\to\bar z$,
there is a convergent sequence $\{w_l\}_{l\in\N}\subset\R^m$ and a subsequence $\{z_{k_l}\}_{l\in\N}$
such that
$w_l\in \Upsilon(z_{k_l})$ holds for all $l\in\N$. Clearly, whenever $\Upsilon$
is locally bounded at $\bar z$, it is inner semicompact w.r.t.\ $\dom\Upsilon$ at
this point.
Similarly, $\Upsilon$ is called inner semicontinuous at $(\bar z,\bar w)\in\gph\Upsilon$
w.r.t.\ $\Omega$ if for each sequence $\{z_k\}_{k\in\N}\subset\Omega$ satisfying $z_k\to\bar z$,
there exists a sequence $\{w_k\}_{k\in\N}\subset\R^m$ such that $w_k\to\bar w$ and
$w_k\in\Upsilon(z_k)$ for all sufficiently large $k\in\N$.

We call $\Upsilon$ convex whenever the property
\[
	\forall z_1,z_2\in\R^n,\,\forall \alpha\in[0,1]\colon\quad
		\alpha \Upsilon(z_1)+(1-\alpha)\Upsilon(z_2)
		\subset
		\Upsilon(\alpha z_1+(1-\alpha)z_2)
\]
holds, i.e., whenever $\gph\Upsilon$ is convex.
This property particularly holds for set-valued mappings of the form
$\Upsilon(z):=g(z)-C$, $z\in\R^n$, where $C\subset\R^m$ is a closed, convex cone and 
the single-valued mapping $g\colon\R^n\to\R^m$ is convex w.r.t.\ $C$, i.e.,
\[
	\forall z_1,z_2\in\R^n,\,\forall \alpha\in[0,1]\colon\quad
		g(\alpha z_1+(1-\alpha)z_2)-\alpha g(z_1)-(1-\alpha)g(z_2)\in C.
\]

Let us recall some essential Lipschitzian properties of set-valued mappings.
Therefore, we fix a point $(\bar z,\bar y)\in\gph\Upsilon$. Recall that $\Upsilon$
possesses the \emph{Aubin property} at $(\bar z,\bar w)$ whenever there are
a constant $\kappa>0$ and neighborhoods $U$ of $\bar z$ and $V$ of $\bar w$ such that
\[
	\forall z_1,z_2\in U\,\forall w\in\Upsilon(z_1)\cap V\colon\quad
	\dist(w,\Upsilon(z_2))\leq\kappa\norm{z_1-z_2}{}.
\]
Furthermore, $\Upsilon$ is called \emph{metrically regular} at $(\bar z,\bar w)$
whenever there are $\kappa>0$ as well as neighborhoods $U$ and $V$ of $\bar z$ and
$\bar w$, respectively, such that
\[
	\forall z\in U,\,\forall w\in V\colon\quad
	\dist(z,\Upsilon^{-1}(w))\leq\kappa\dist(w,\Upsilon(z)).
\]
It is well known that $\Upsilon$ is metrically regular at $(\bar z,\bar w)$ if and
only if $\Upsilon^{-1}$ possesses the Aubin property at $(\bar w,\bar z)$.
Finally, recall that $\Upsilon$ is referred to as \emph{metrically subregular}
at $(\bar z,\bar w)$ whenever there exist a constant $\kappa>0$ and a neighborhood $U$
of $\bar z$ satisfying
\[
	\forall z\in U\colon\quad
	\dist(z,\Upsilon^{-1}(\bar w))\leq\kappa\dist(\bar w,\Upsilon(z)).
\]
The infimum of all such constants $\kappa$ is referred to as the modulus of metric
subregularity.
Let us mention that $\Upsilon$ is metrically subregular at $(\bar z,\bar w)$ if and
only if the inverse $\Upsilon^{-1}$ is so-called \emph{calm} at $(\bar w,\bar z)$,
see \cite{HenrionJouraniOutrata2002,HenrionOutrata2005} for a definition and further details.
It has turned out that metric subregularity is of essential importance in
order to guarantee applicability of the prominent
calculus for the limiting constructions of
generalized differentiation.
This is one of the reasons why sufficient conditions for its presence were studied,
see e.g.\ 
\cite{BaiYeZhang2019,BenkoCervinkaHoheisel2019,FabHenKruOut10,GfrererKlatte2016,IofOut08,ZheNg10}
and the references therein as well as
the aforementioned papers dealing with calmness.
We would like to mention that each polyhedral set-valued mapping,
i.e., a set-valued mapping whose graph is the union of finitely many convex polyhedral sets,
is metrically subregular at each point of its graph.
This result dates back to \cite{Robinson1981}.

\subsubsection*{Variational analysis}

For a closed set $A\subset\R^n$ and some point $\bar z\in A$, we refer to
\[
	\mathcal T_A(\bar z)
	:=
	\left\{
		d\in\R^n\,\middle|\,
		\exists\{z_k\}_{k\in\N}\subset A,\,
		\exists\{t_k\}_{k\in\N}\subset \R_+\colon\,
		z_k\to\bar z,\,t_k\downarrow 0,\,(z_k-\bar z)/t_k\to d
	\right\}
\]
as the (Bouligand) tangent cone to $A$ at $\bar z$.
Based on that, let us introduce
\begin{align*}
	\widehat{\mathcal N}_A(\bar z)
	&:=
	\mathcal T_A(\bar z)^\circ,\\
	\mathcal N_A(\bar z)
	&:=	\left\{
			\xi\in\R^n\,\middle|\,
			\begin{aligned}
				&\exists \{z_k\}_{k\in\N}\subset A,\,
				 \exists \{\xi_k\}_{k\in\N}\subset\R^n\colon\\
				&\qquad 
					z_k\to\bar z,\,\xi_k\to\xi,\,
					\xi_k\in\widehat{\mathcal N}_A(z_k)\,\forall k\in\N
			\end{aligned}
		\right\},
\end{align*}
the so-called regular (or Fr\'{e}chet) and limiting (or Mordukhovich) normal
cone to $A$ at $\bar z$, respectively. In case where $A$ is convex, these normal cones correspond
to the normal cone in the sense of convex analysis, i.e.,
$\widehat{\mathcal N}_A(\bar z)=\mathcal N_A(\bar z)=(A-\bar z)^\circ$.

For a locally Lipschitz continuous function $\varphi\colon\R^n\to\R$, we define its
limiting subdifferential at some point $\bar z\in\R^n$ by means of
\[
	\partial\varphi(\bar z)
	:=
	\left\{
		\xi\in\R^n\,\middle|\,
		(\xi,-1)\in\mathcal N_{\epi\varphi}(\bar z,\varphi(\bar z))
	\right\}
\]
where $\epi\varphi:=\{(z,s)\in\R^n\times\R\,|\,s\geq\varphi(z)\}$ denotes the epigraph
of $\varphi$. In case where $\varphi$ is convex, the above arguments yield that $\partial\varphi(\bar z)$
coincides with the subdifferential in the sense of convex analysis, i.e.,
\[
	\partial\varphi(\bar z)
	=
	\bigl\{
		\xi\in\R^n\,\big|\,
		\forall z\in\R^n\colon\,\varphi(z)\geq\varphi(\bar z)+\xi^\top(z-\bar z)
	\bigr\}
\]
holds in this case.

For a set-valued mapping $\Upsilon\colon\R^n\tto\R^m$ with closed graph and some
point $(\bar z,\bar w)\in\gph\Upsilon$, the set-valued mapping 
$D^*\Upsilon(\bar z,\bar w)\colon\R^m\tto\R^n$ given by
\[
	\forall \eta\in\R^m\colon\quad
	D^*\Upsilon(\bar z,\bar w)(\eta)
	:=
	\left\{
		\xi\in\R^n\,\middle|\,
		(\xi,-\eta)\in\mathcal N_{\gph\Upsilon}(\bar z,\bar w)
	\right\}
\]
is referred to as the (limiting) coderivative of $\Upsilon$ at $(\bar z,\bar w)$.
For a single-valued mapping $\upsilon\colon\R^n\to\R^m$, we exploit the notation
$D^*\upsilon(\bar z):=D^*\upsilon(\bar z,\upsilon(\bar z))$. Whenever $\upsilon$
is continuously differentiable at $\bar z$, 
then $D^*\upsilon(\bar z)(\eta)=\nabla\upsilon(\bar z)^\top\eta$
holds for all $\eta\in\R^m$ where $\nabla\upsilon(\bar z)$ denotes the Jacobian of $\upsilon$
at $\bar z$.

Let us recall that $\Upsilon$ possesses the Aubin property at $(\bar z,\bar w)$ if and
only if the condition
\[
	D^*\Upsilon(\bar z,\bar w)(0)=\{0\}
\]
holds. Thus, $\Upsilon$ is metrically regular at $(\bar z,\bar w)$
if and only if we have
\[
	\ker D^*\Upsilon(\bar z,\bar w)=\{0\}.
\]
Both criteria are referred to as \emph{Mordukhovich criterion} in the literature,
see \cite[Theorem~3.3]{Mordukhovich2018}.

\subsection{M-stationarity for optimization problem with generalized equation constraints}
	\label{sec:abstract_M_St}

The following result can be found in \cite[Theorem~4.1]{HenrionJouraniOutrata2002} 
and \cite[Proposition~4.1]{GfrererOutrata2016b}. It provides a calculus rule for
the computation of the limiting normal cone to the pre-image associated with
a set-valued mapping which is metrically subregular at a certain point of interest.
\begin{proposition}\label{Pro:MSRMain}
 	Let $\Upsilon\colon\R^n\tto\R^m$ be a set-valued mapping having locally closed graph 
 	around $(\bar z,\bar w)\in\gph \Upsilon$
 	and assume that $\Upsilon$ is metrically subregular at $(\bar z,\bar w)$ with modulus $\bar\kappa$. 
 	Then the following estimate is valid for each
 	$\kappa > \bar\kappa$:
 	\begin{equation}\label{eq:NCtoM}
  		\mathcal N_{\Upsilon^{-1}(\bar w)}(\bar z) 
  		\subset
  		\left\{
  			\xi\in\R^n \,\middle|\,
  			\exists \eta \in \kappa \norm{\xi}{} \B\colon\,
  			\xi \in D^* \Upsilon(\bar z,\bar w)(\eta) 
  		\right\}.
 	\end{equation}
\end{proposition}

This motivates the subsequently stated extension of the definition of M-stationarity to abstract 
optimization programs of the form
\begin{equation}\label{eq:general_program}\tag{G}
	\begin{split}
		\varphi(z)&\,\to\,\min\\
		0&\,\in \, \Upsilon(z),
	\end{split}
\end{equation}
where $\varphi\colon\R^n\to\R$ is a locally Lipschitz continuous function and
$\Upsilon\colon\R^n\tto\R^m$ is a set-valued mapping with closed graph,
observing that $\Upsilon^{-1}(0)$ is the feasible set of \eqref{eq:general_program}.
A similar stationarity condition has been studied in the literature, 
see \cite[Section~3]{YeYe1997}.
\begin{definition}\label{Def:Mstat}
We say that a feasible point $\bar z \in \Upsilon^{-1}(0)$ of the program \eqref{eq:general_program} 
is M-stationary provided there exists $\eta \in \R^m$ such that
\[
 	0 \in \partial \varphi(\bar z) + D^*\Upsilon(\bar z,0)(\eta).
\]
\end{definition}

In case where the constraint mapping $\Upsilon$ is given in the form of a
so-called ``feasibility'' mapping, i.e., $\Upsilon(z) := \Phi(z) - \Omega$
for all $z\in\R^n$,
where, for simplicity, $\Phi\colon\R^n\to\R^m$ is a continuously differentiable mapping 
and $\Omega\subset\R^m$ is a closed set, \cref{Pro:MSRMain} yields the standard pre-image calculus rule
\[
	\mathcal N_{\Phi^{-1}(\Omega)}(z) \subset \nabla\Phi(z)^\top \mathcal N_{\Omega}(\Phi(z))
\]
for each $z\in\R^n$ such that $\Upsilon$ is metrically subregular at $(z,0)\in\gph\Upsilon$.
This also justifies the M-stationarity terminology in \cref{Def:Mstat}.
Moreover, the additional bound $\norm{\eta}{} \leq \kappa \norm{\xi}{}$ in \eqref{eq:NCtoM} implies that for every
$\xi \in \mathcal N_{\Phi^{-1}(\Omega)}(z)$, there exists a multiplier 
$\eta \in \mathcal N_{\Omega}(\Phi(z)) \cap \kappa \norm{\xi}{} \B$
with $\xi = \nabla\Phi(z)^\top \eta$.
This observation, however, brings forth the following important property of the associated multiplier 
mapping which has been discussed in \cite[Theorem 3.9]{Benko2019} recently.
\begin{proposition}\label{prop:inner_semicompactness_of_multiplier_map}
	Let $\Phi\colon\R^n\to\R^m$ be continuously differentiable and let $\Omega\subset\R^m$ be closed.
	We consider $\Upsilon\colon\R^n\tto\R^m$ given by $\Upsilon(z):=\Phi(z)-\Omega$ for each $z\in\R^n$.
 	Let $(\bar z,\bar z^*) \in \gph \mathcal N_C$ for $C := \Phi^{-1}(\Omega)$ be fixed and assume 
 	that the feasibility mapping
 	$\Upsilon$ is metrically subregular at $(\bar z,0)$.
 	Then the multiplier mapping $\Lambda\colon\R^n \times \R^n \tto \R^m$, given by
	 \begin{equation*}
	 	\forall (z,z^*)\in\R^n\times\R^n\colon\quad
  		\Lambda( z,z^*):= 
  		\bigl\{
  			\eta \in \mathcal N_{\Omega}(\Phi(z))\,\big|\,
  			\nabla\Phi(z)^\top \eta = z^*
  		\bigr\},
 	\end{equation*}
 	is inner semicompact at $(\bar z,\bar z^*)$ w.r.t.\ $\gph \mathcal N_C$.
 	Moreover, $\gph \mathcal N_C = \dom \Lambda$ holds provided $\Omega$ is convex.
\end{proposition}
For comparison, it is well known that the stronger assumption of metric regularity
yields local boundedness of the multiplier mapping $\Lambda$, and this has indeed been often used
to guarantee its inner semicompactness, see \cref{sec:app_bilevel_programming} as well.
While it may seem that local boundedness is not much stronger than inner semicompactness,
when the underlying assumptions are compared, the gap is noteworthy.
In case of standard nonlinear programs, metric regularity equals the prominent
\emph{Mangasarian--Fromovitz constraint qualification} (MFCQ), while in case
of the more general class of geometric constraints as discussed in 
\cref{prop:inner_semicompactness_of_multiplier_map}, metric regularity of the 
feasibility mapping $\Upsilon$ at $(\bar z,0)\in\gph\Upsilon$ boils down to
\[
	\nabla\Phi(\bar z)^\top\eta=0,\,\eta\in\mathcal N_\Omega(\Phi(\bar z))
	\quad\Longrightarrow\quad
	\eta=0,
\] 
which is called
\emph{generalized Mangasarian--Fromovitz constraint qualification} (GMFCQ)
or \emph{no nonzero abnormal multiplier constraint qualification} (NNAMCQ).
The above propositions often enable us to relax the metric regularity requirement
to the so-called \emph{metric subregularity constraint qualification} (MSCQ) demanding 
$\Upsilon$ to be metrically subregular at $(\bar z,0)$.

\subsection{Calculus of set-valued mappings}\label{sec:calculus_of_set_valued_maps}

Throughout this section, we assume that all considered set-valued
mappings possess closed graphs.

We first review the chain rule for the computation of coderivatives of
compositions. Therefore, let us consider set-valued mappings 
$S_1\colon\R^n\tto\R^m$ and $S_2\colon\R^m\tto\R^s$ as well as their
composition $S_2\circ S_1\colon\R^n\tto\R^s$ given by
\[
	\forall z\in\R^n\colon\quad
	(S_2\circ S_1)(z):=\bigcup\limits_{y\in S_1(z)} S_2(y).
\]
In order to estimate the coderivative of $S_2\circ S_1$,
we introduce the standard
``intermediate'' mapping $\Xi\colon\R^n\times\R^s\tto\R^m$ given by
\begin{equation}\label{eq:intermediate_mapping}
	\forall z\in\R^n,\,\forall w\in\R^s\colon\quad
	\Xi(z,w):=S_1(z)\cap S_2^{-1}(w)=\{y\in S_1(z)\,|\,w\in S_2(y)\},
\end{equation}
together with the ``feasibility'' mapping $\Upsilon\colon\R^n\times\R^s\times\R^m\tto\R^{m}\times\R^{s}$
in the following form:
\begin{equation*}
	\forall z\in\R^n,\,\forall w\in\R^s,\,\forall y\in\R^m\colon\quad
	\Upsilon(z,w,y):=
	\begin{pmatrix}
		S_1(z)-y\\
		S_2(y)-w
	\end{pmatrix}.
\end{equation*}
The following chain rule is taken from \cite[Theorem~5.2]{BenkoMehlitz2020}.

\begin{lemma}\label{lem:chain_rule}
	Fix $(\bar z,\bar w)\in\gph(S_2\circ S_1)$.
	Let $\Xi$ be inner semicompact w.r.t.\ $\dom\Xi$ at $(\bar z,\bar w)$
	and let $\Upsilon$ be metrically subregular at all points
	$((\bar z,\bar w,\bar y),(0,0))$ such that $\bar y\in\Xi(\bar z,\bar w)$.
	Then we have
	\[
		\forall w^*\in\R^s\colon\quad
		D^*(S_2\circ S_1)(\bar z,\bar w)(w^*)
		\subset
		\bigcup\limits_{\bar y\in\Xi(\bar z,\bar w)}
			\bigl(D^*S_1(\bar z,\bar y)\circ D^*S_2(\bar y,\bar w)\bigr)(w^*).
	\]
\end{lemma}

In \cite[Theorem~5.2]{BenkoMehlitz2020},
the subregularity assumption is replaced by the equivalent
calmness assumption associated with the inverse mapping of $\Upsilon$ given by
\[
	\Upsilon^{-1}(a,b)=\{(z,w,y) \mv y + a \in S_1(z), w + b \in S_2(y)\}.
\]
Furthermore, note that $\Upsilon$ is slightly different from
the typical feasibility mapping used in the chain rule,
see e.g.\ \cite[proof of Theorem 10.37]{RockafellarWets1998}.
However, it is worth to mention that the respective metric 
subregularity assumptions, which need to be postulated
on these maps in order to infer the chain rule, are actually equivalent
due to \cite[Proposition~2.6]{BenkoMehlitz2020}.

The proof of \cref{lem:chain_rule} is based on the essential relations
\begin{equation}\label{eq:chain_rule_setting}
    \gph (S_2\circ S_1) = \dom \Xi,\qquad\qquad \gph \Xi = \Upsilon^{-1}(0,0).
\end{equation}
Roughly speaking, the inner semicompactness of $\Xi$ provides a connection between
the graph of $S_2\circ S_1$ and the graph of $\Xi$ while the metric subregularity 
of $\Upsilon$ connects the graph of $\Xi$ with the graph of $\Upsilon$ via \cref{Pro:MSRMain},
see \cite[Sections~3 and 5.3]{BenkoMehlitz2020} for details.
The combination of these two assumptions, hence, yields a way to estimate
the coderivative of $S_2\circ S_1$ via the coderivative of $\Upsilon$, namely
\begin{align*}
	&z^* \in D^*(S_2\circ S_1)(\bar z,\bar w)(w^*)\\
	&\quad \Longrightarrow 
	\exists \bar y \in \Xi(\bar z,\bar w),\,\exists\xi_1\in\R^m,\,\exists\xi_2\in\R^s\colon
	\,
	(z^*,-w^*,0) \in D^*\Upsilon((\bar z, \bar w, \bar y),(0,0))(\xi_1,\xi_2).
\end{align*}
Finally, the coderivative of $\Upsilon$ can be expressed via its components as
\[
	D^*\Upsilon((\bar z, \bar w, \bar y),(0,0))(\xi_1,\xi_2) 
	=
	D^*S_1(\bar z,\bar y)(\xi_1) \times \{0\} \times D^*S_2(\bar y,\bar w)(\xi_2)
	\ - \
	(0,\xi_2,\xi_1).\]
This can be seen from \cref{lem:equality_coderivative_products} (b)
(whose proof does not depend on the chain rule)
where one can split $(z,w,y)$ either as $((z,w),y)$
or as $(z,(w,y))$.
The above formulas now readily yield the desired chain rule via
the coderivatives of $S_1$ and $S_2$.

Naturally, in specific situations, one may prefer to use different
intermediate and feasibility mappings satisfying \eqref{eq:chain_rule_setting}.
The question then remains whether one can handle the coderivative
of the chosen feasibility mapping to derive the chain rule without
additional assumptions.

For set-valued mappings $\Gamma_1\colon\R^n\tto\R^{m_1}$ and $\Gamma_2\colon\R^n\tto\R^{m_2}$,
we consider the product mapping $\Gamma\colon\R^n\tto\R^{m_1}\times\R^{m_2}$ given by
\[
	\forall z\in\R^n\colon\quad
		\Gamma(z):=\Gamma_1(z)\times\Gamma_2(z).
\]
For later use, we need to compute or at least estimate the coderivative of this mapping. 
From \cite[Section~5.4]{BenkoMehlitz2020}, we obtain the next result.

\begin{lemma}\label{lem:product_rule}
	Fix $(\bar z,(\bar w_1,\bar w_2))\in\gph\Gamma$.
	Assume that $\widetilde\Upsilon\colon \R^n\times\R^{m_1}\times\R^{m_2}\times\R^n\times\R^n 
	\tto \R^n\times\R^n\times\R^{m_1}\times\R^{m_2}$,
	given by
	\[
		\forall z,q_1,q_2\in\R^n,\,\forall w_1\in\R^{m_1},\,\forall w_2\in\R^{m_2}\colon\quad
				\widetilde\Upsilon(z,w_1,w_2,q_1,q_2):=
					\begin{pmatrix}
                            z-q_1 \\
                            z-q_2	\\
                            \Gamma_1(q_1) - w_1 \\
                            \Gamma_2(q_2) - w_2
                    \end{pmatrix},
	\]
	is metrically subregular at $((\bar z,\bar w_1,\bar w_2,\bar z,\bar z),(0,0,0,0))$.
	Then we have
	\begin{equation}\label{eq:coderivative_product_rule}
		D^*\Gamma(\bar z,(\bar w_1,\bar w_2))(\xi_1,\xi_2)
		\subset 
		D^*\Gamma_1(\bar z,\bar w_1)(\xi_1)+D^*\Gamma_2(\bar z,\bar w_2)(\xi_2).
	\end{equation}
	Moreover, $\widetilde{\Upsilon}$ is metrically subregular
	at $((\bar z,\bar w_1,\bar w_2,\bar z,\bar z),(0,0,0,0))$
	whenever $\Gamma_1$ and $\Gamma_2$ are polyhedral or
	if the qualification condition
	\begin{equation}\label{eq:QualCond_MC}
				D^*\Gamma_1(\bar z,\bar w_1)(0)
				\cap
				\bigl(-D^*\Gamma_2(\bar z,\bar w_2)(0)\bigr)
				=
				\{0\}
	\end{equation}
	holds, in which case $\widetilde{\Upsilon}$ is even metrically regular at the point of interest.
\end{lemma}

The qualification condition \eqref{eq:QualCond_MC} is clearly satisfied
if one of the mappings $\Gamma_i$, $i = 1, 2$, possesses the Aubin property at
$(\bar z,\bar w_i)$. In the following lemmas,
we review more results from \cite[Section~5.4]{BenkoMehlitz2020}
concerning some exemplary settings where \eqref{eq:QualCond_MC} holds naturally.
\begin{lemma}\label{lem:coderivative_product}
	Fix $(\bar z,(\bar w_1,\bar w_2))\in\gph\Gamma$ and assume 
	that one of the following conditions holds.
	Then \eqref{eq:QualCond_MC} is satisfied and, thus, the estimate
	\eqref{eq:coderivative_product_rule} is valid.
	\begin{enumerate}
	 \item[(a)] There are a locally Lipschitz continuous function 
            $\gamma\colon\R^n\to\R^{m_2}$
			as well as a closed set $\Omega\subset\R^{m_2}$ such that $\Gamma_2$ is given by
			$\Gamma_2(z):=\gamma(z)-\Omega$ for all $z\in\R^n$
			(an analogous statements holds if $\,\Gamma_1$ admits
			such a representation).
	 \item[(b)] The variables $z$ can be decomposed as $z=(z_1,z_2)\in\R^{n_1}\times\R^{n_2}$.
	 	Furthermore, there are set-valued mappings
	 	$\widetilde\Gamma_1\colon\R^{n_1}\tto\R^{m_1}$ and $\widetilde\Gamma_2\colon\R^{n_2}\tto\R^{m_2}$
	 	as well as locally Lipschitz continuous functions
	 	$\gamma_1\colon\R^{n_2}\to\R^{m_1}$ and $ \gamma_2\colon\R^{n_1}\to\R^{m_2}$
	 	such that
	 	\[
	 		\forall z=(z_1,z_2)\in\R^{n_1}\times\R^{n_2}\colon\quad
	 		\Gamma_1(z):=\widetilde\Gamma_1(z_1) +  \gamma_1(z_2),
	 		\quad
	 		\Gamma_2(z):=\widetilde\Gamma_2(z_2)  + \gamma_2(z_1).
	 	\]
	\end{enumerate}
\end{lemma}
We note that, in principle, the setting (a) can be viewed as an extreme
case of setting (b) with $n_1=n$ and $n_2=0$.
Naturally, this does not fit formally, but one could regard
a set as a set-valued mapping from a zero-dimensional space
and propose a suitable relation for its coderivative via
the normal cone to the set.

In the specific settings outlined above, 
we can, in fact, get equality in \eqref{eq:coderivative_product_rule}
instead of inclusion if we strengthen the Lipschitzness
of the single-valued parts to continuous differentiability,
see \cite[Lemma~5.7]{BenkoMehlitz2020}.
\begin{lemma}\label{lem:equality_coderivative_products}
 	Fix $(\bar z,(\bar w_1,\bar w_2))\in\gph\Gamma$.
 	Then the following statements hold.
 	\begin{enumerate}
  \item[(a)] Consider the setting (a) from \cref{lem:coderivative_product} and assume
  		that the function $\gamma$ is continuously differentiable. Then we have
  		\[
  			D^*\Gamma(\bar z,(\bar w_1,\bar w_2))(\xi_1,\xi_2) =
  				\begin{cases}
            		D^*\Gamma_1(\bar z,\bar w_1)(\xi_1) + \nabla\gamma(\bar z)^\top\xi_2
            			& \text{if } \xi_2 \in \mathcal N_{\Omega}(\gamma(\bar z)-\bar w_2), \\
            		\varnothing 
            			& \text	{otherwise.}
            	\end{cases}
  		\]
  \item[(b)] Consider the setting (b) from \cref{lem:coderivative_product} and assume
  		that the functions 	$\gamma_1$ and $\gamma_2$ are continuously differentiable.
	 	Then we have
  		\[
  			D^*\Gamma(\bar z,(\bar w_1,\bar w_2))(\xi_1,\xi_2) =
	 					\begin{pmatrix}
                            D^*\widetilde\Gamma_1(\bar z_1,\bar w_1-\gamma_1(\bar z_2))(\xi_1) 
                            	+ \nabla\gamma_2(\bar z_1)^\top\xi_2 \\
                            D^*\widetilde\Gamma_2(\bar z_2,\bar w_2-\gamma_2(\bar z_1))(\xi_2) 
                            	+ \nabla\gamma_1(\bar z_2)^\top\xi_1
                        \end{pmatrix}.
        \]
 \end{enumerate}
\end{lemma}

Up to now, we discussed the auxiliary issue of how to
guarantee metric subregularity of $\widetilde{\Upsilon}$,
while the main target is to justify the upper estimate for the coderivative of $\Gamma$.
Later, we will also employ metric subregularity of the product mapping $\Gamma$ itself, 
which is, however, often a more difficult issue.
We conclude the preliminary part by a simple result regarding
how to ensure metric subregularity of $\Gamma$ from
metric subregularity of its factors $\Gamma_1$ and $\Gamma_2$.
Let us mention that there exists a much deeper result by Klatte and Kummer, see
\cite[Theorem~3.6]{KlatteKummer2002}, on calmness of an intersection of mappings,
which can be equivalently stated in terms of metric subregularity of a product of mappings,
and which has several very interesting applications.
In particular, since $\widetilde{\Upsilon}$ also has the structure of a product,
it can be utilized even in situations when \eqref{eq:QualCond_MC} fails to hold.
Moreover, as demonstrated in \cite{AdamHenrionOutrata2018}, it applies quite nicely
to the setting discussed in \cref{sec:bilevel_programming}.
Nevertheless, for the purposes of this paper, we will only need the following result.

\begin{lemma}\label{lem:metric_subregularity_of_productsII}
	Fix $(\bar z,(\bar w_1,\bar w_2))\in\gph\Gamma$ and assume that
	$\Gamma_i$ is metrically subregular at $(\bar z,\bar w_i)$ for $i=1,2$
	and, moreover, that the mapping $\Sigma\colon\R^n\tto\R^{n}\times\R^{n}$ given by
			\begin{equation}\label{eq:mapping_Sigma}
				\forall z\in\R^n\colon\quad
				\Sigma (z):= (z - \Gamma_1^{-1}(\bar w_1)) \times
                            (z - \Gamma_2^{-1}(\bar w_2))
            \end{equation}
			is metrically subregular at $(\bar z,(0,0))$.
			Then $\Gamma$ is metrically subregular at $(\bar z,(\bar w_1,\bar w_2))$.
\end{lemma}
\begin{proof}
	By assumption, we find neighborhoods $U,U_1,U_2\subset\R^n$ of $\bar z$
	and reals $\kappa,\kappa_1,\kappa_2>0$ satisfying 
	\[
		\forall z\in U\colon\quad
		\dist(z,\Sigma^{-1}(0,0))\leq\kappa\dist((0,0),\Sigma(z))
	\]
	as well as
	\[
		\forall z\in U_i\colon\quad
		\dist(z,\Gamma_i^{-1}(\bar w_i))\leq\kappa_i\dist(\bar w_i,\Gamma_i(z))
	\]
	for $i=1,2$.
    Due to $\Gamma^{-1}(\bar w_1,\bar w_2)=\Sigma^{-1}(0,0)$, we obtain
    \begin{align*}
     	\dist(z,\Gamma^{-1}(\bar w_1,\bar w_2)) 
     	& \leq 
     	\kappa 	\big( 	\dist(z,\Gamma_1^{-1}(\bar w_1)) +
     					\dist(z,\Gamma_2^{-1}(\bar w_2))
     			\big) 
     			\\
     	& \leq 
     	 \kappa \max\{\kappa_1,\kappa_2\} \dist((\bar w_1,\bar w_2),\Gamma(z))
    \end{align*}
    for each $z\in U\cap U_1\cap U_2$, and this shows the claim.
\end{proof}

Since $\Sigma$ from \eqref{eq:mapping_Sigma} has a simpler structure than $\Gamma$,
its coderivative can be readily computed by \cref{lem:equality_coderivative_products}.
Particularly, we note that metric regularity of $\Sigma$ at $(\bar z,(0,0))$ reads
\[
	\mathcal N_{\Gamma^{-1}_1(\bar w_1)}(\bar z)
	\cap
	\left(-\mathcal N_{\Gamma^{-1}_2(\bar w_2)}(\bar z)\right)
	=
	\{0\},
\]
and, by metric subregularity of $\Gamma_i$ at $(\bar z,\bar w_i)$, $i=1,2$, 
as well as the pre-image rule from \cref{Pro:MSRMain}, the condition
\[
	\rge D^*\Gamma_1(\bar z,\bar w_1)
	\cap
	\left(-\rge D^*\Gamma_2(\bar z,\bar w_2)\right)
	=
	\{0\}
\] 
is sufficient to guarantee its validity.

\section{Abstract analysis with and without implicit variables}
	\label{sec:lambda_as_variable}

In this section, we first address the relationship between the 
minimizers of \eqref{eq:implicit_multipliers} and \eqref{eq:explicit_multipliers},
respectively. Afterwards, we will deal with the derivation of necessary optimality
conditions and constraint qualifications for \eqref{eq:implicit_multipliers} with
the aid of limiting variational analysis. In order to infer optimality conditions
in terms of initial problem data (particularly, in terms of the generalized
derivatives of $F$ and $G$), we use the chain and product rule of coderivative calculus
from \cref{sec:calculus_of_set_valued_maps}.
As we will see, the associated assumptions for their respective application are valid
in several practically relevant settings, which shows that the essential
constraint qualification boils down to one subregularity assumption.
We close the section with some brief remarks
regarding sufficiency of the derived optimality conditions for convex problems of type \eqref{eq:implicit_multipliers}.

\subsection{Solution behavior}\label{sec:relationship_of_P_and_Q}

Here, we want to study the relationship between the local and global minimizers
of \eqref{eq:implicit_multipliers} and \eqref{eq:explicit_multipliers}, respectively.
To this end, the intermediate mapping $K\colon\R^n\tto\R^m$ given by
\[
	\forall z\in\R^n\colon\quad
	K(z):=\{\lambda\in F(z)\,|\,0\in G(z,\lambda)\}
\]
will be of essential importance. Note that we have $Z=M\cap\dom K$ by definition.

\begin{lemma}\label{lem:H_closed_graph}
	If $F$ is a locally bounded set-valued mapping, then $H$ possesses a closed graph.
\end{lemma}
\begin{proof}
	Let $\{(z_k,w_k)\}_{k\in\N}\subset\gph H$ be a sequence converging to some
	$(\bar z,\bar w)\in\R^n\times\R^s$. Then, by definition of $H$, for each $k\in\N$, we find
	$\lambda_k\in F(z_k)$ such that $w_k\in G(z_k,\lambda_k)$ holds. Since
	$F$ is locally bounded at $\bar z$, $\{\lambda_k\}_{k\in\N}$ must be bounded.
	Thus, along a subsequence (without relabeling) it converges to some $\bar\lambda\in\R^m$.
	Recalling that $F$ and $G$ are set-valued mappings with a closed graph,
	$\bar\lambda\in F(\bar z)$ and $\bar w\in G(\bar z,\bar\lambda)$ follow, i.e.,
	$\bar w\in H(\bar z)$ is obtained. This shows the closedness of $\gph H$.
\end{proof}

Exemplary, let us mention that in the setting of \cref{sec:bilevel_programming},
the mapping $F$ is cone-valued and not likely to be locally bounded. However, one
can check that the associated mapping $H$ still possesses a closed graph by closedness of
the normal cone mapping associated with a convex set.

Subsequently, we discuss some properties of the mapping $K$.
\begin{lemma}\label{lem:inner_semicompactness_of_K}
	The following assertions hold.
	\begin{enumerate}
		\item[(a)] The mapping $K$ possesses a closed graph.
		\item[(b)] Fix $\bar z\in\dom K$. If $F$ is locally bounded at $\bar z$, then
			$K$ is inner semicompact at $\bar z$ w.r.t.\ $\dom K$.
	\end{enumerate}
\end{lemma}
\begin{proof}
	Let us start with the proof of assertion (a).
	Let $\{(z_k,\lambda_k)\}_{k\in\N}\subset\gph K$ be chosen such that $z_k\to\bar z$
		and $\lambda_k\to\bar\lambda$ hold.
		By definition, we have $\lambda_k\in F(z_k)$ and $0\in G(z_k,\lambda_k)$ for each $k\in\N$.
		The closedness of $\gph F$ and $\gph G$ yields $\bar\lambda\in F(\bar z)$
		and $0\in G(\bar z,\bar\lambda)$, respectively.
		Hence, we have $\bar\lambda\in K(\bar z)$, i.e., $\gph K$ is closed.\\
	In order to show (b), let $\{z_k\}_{k\in\N}\subset\dom K$ be chosen such that $z_k\to\bar z$ holds.
		Then we find a sequence $\{\lambda_k\}_{k\in\N}\subset\R^m$ such that $\lambda_k\in K(z_k)$
		holds for all $k\in\N$. Noting that $K$ is locally bounded at $\bar z$ since $F$
		enjoys this property by assumption, $\{\lambda_k\}_{k\in\N}$
		is bounded and possesses a convergent subsequence. 
		Thus, $K$ is inner semicompact at $\bar z$ w.r.t.\ $\dom K$.
\end{proof}

In the upcoming theorem, we take a look at the global minimizers 
of \eqref{eq:implicit_multipliers} and \eqref{eq:explicit_multipliers}.
\begin{theorem}
	\label{thm:global_relationship}\
	\begin{enumerate}
		\item[(a)] Let $\bar z\in\R^n$ be a global minimizer of \eqref{eq:implicit_multipliers}.
			Then, for each $\lambda\in K(\bar z)$, $(\bar z,\lambda)$ is a global
			minimizer of \eqref{eq:explicit_multipliers}.
		\item[(b)] Let $(\bar z,\bar\lambda)\in\R^n\times\R^m$ be a global minimizer of
			\eqref{eq:explicit_multipliers}. Then $\bar z$ is a global minimizer
			of \eqref{eq:implicit_multipliers}.
	\end{enumerate}
\end{theorem}
\begin{proof}
	Both statements of the theorem follow from the observation that for given $z\in \R^n$,
	we have the equivalences
	\[
		z\in Z\,
		\Longleftrightarrow
		\,z\in M\cap\dom K\,
		\Longleftrightarrow
		\,\exists\lambda\in\R^m\colon\,(z,\lambda)\in\tilde Z,
	\]
	while the objective function of \eqref{eq:explicit_multipliers} does not depend
	on the variable $\lambda$.
\end{proof}

The above theorem shows that the relationship between \eqref{eq:implicit_multipliers}
and \eqref{eq:explicit_multipliers} is straight whenever global minimizers are under
consideration. Thus, whenever these problems actually do not possess local minimizers
which are not globally optimal, then \eqref{eq:implicit_multipliers} and
\eqref{eq:explicit_multipliers} are \emph{equivalent} w.r.t.\ their minimizers in
the sense of \cref{thm:global_relationship}.
\begin{corollary}\label{cor:convex_programs}
	Consider the situation where $F$ and $G$ are convex set-valued mappings.
	One can easily check that this implies convexity of $H$.
	Particularly, whenever $f$ is a convex function and $M$ is a convex set,
	then \eqref{eq:implicit_multipliers} and \eqref{eq:explicit_multipliers}
	are both convex optimization problems (i.e., their objective function
	and feasible set are convex, respectively) whose minimizers correspond to each
	other in the sense of \cref{thm:global_relationship}.
\end{corollary}

A quite popular and nearby situation which is addressed by \cref{cor:convex_programs} arises
in the context of (conic) linear programming where slack variables are used to transfer a
given problem instance into standard normal form. It is well known that the original problem
and its surrogate are equivalent w.r.t.\ their minimizers.

Next, we investigate the relationship between local minimizers of \eqref{eq:implicit_multipliers} and 
\eqref{eq:explicit_multipliers}. As we will see, this issue is much more delicate in 
comparison with the situation where globally optimal solutions of these programs are under
consideration.
\begin{theorem}
	\label{thm:local_relationship}\
	\begin{enumerate}
		\item[(a)] Let $\bar z\in\R^n$ be a local minimizer of \eqref{eq:implicit_multipliers}.
			Then, for each $\lambda\in K(\bar z)$, $(\bar z,\lambda)$ is a local minimizer
			of \eqref{eq:explicit_multipliers}.
		\item[(b)] Let $(\bar z,\lambda)\in\R^n\times\R^m$ be a local minimizer of
			\eqref{eq:explicit_multipliers} for each $\lambda\in K(\bar z)$.
			Furthermore, let $K$ be inner semicompact at $\bar z$ w.r.t.\ $\dom K$.
			Then $\bar z$ is a local minimizer of \eqref{eq:implicit_multipliers}.
	\end{enumerate}
\end{theorem}
\begin{proof}
	For the proof of (a), suppose that, for some local minimizer $\bar z\in Z$ of
	\eqref{eq:implicit_multipliers}, there is some $\bar\lambda\in K(\bar z)$ 
	such that $(\bar z,\bar\lambda)$ is not a
	local minimizer of \eqref{eq:explicit_multipliers}.
	Then there is a sequence $\{(z_k,\lambda_k)\}_{k\in\N}\subset\tilde Z$ 
	such that $z_k\to\bar z$, $\lambda_k\to\bar\lambda$, and
	$f(z_k)<f(\bar z)$ holds for each $k\in\N$.
	Since we have $z_k\in M\cap\dom K$ by definition of $K$, $z_k$ is feasible to
	\eqref{eq:implicit_multipliers} for each $k\in\N$.
	Thus, $\bar z$ cannot be a local minimizer of \eqref{eq:implicit_multipliers}
	which is a contradiction.
	This shows (a).\\			
	For the proof of (b), let $(\bar z,\lambda)\in\R^n\times\R^m$ be a local minimizer of
			\eqref{eq:explicit_multipliers}
			for each $\lambda\in K(\bar z)$ and let $K$ be inner semicompact at $\bar z$
			w.r.t.\ $\dom K$.
			From $\bar z\in M\cap\dom K$ we infer the feasibility of $\bar z$ to \eqref{eq:implicit_multipliers}.
			Suppose now that $\bar z$ is not a local minimizer of \eqref{eq:implicit_multipliers}.
			Then we find a sequence $\{z_k\}_{k\in\N}\subset Z$ such that $z_k\to\bar z$ 
			while $f(z_k)<f(\bar z)$ is satisfied for all $k\in\N$.
			From $z_k\in Z$ we deduce $z_k\in M\cap\dom K$ for each $k\in\N$.
			Due to the inner semicompactness of $K$ at $\bar z$ w.r.t.\ $\dom K$, 
			there are a subsequence $\{z_{k_l}\}_{l\in\N}$ and a sequence 
			$\{\lambda_{l}\}_{l\in\N}\subset\R^m$ converging to some $\bar\lambda\in\R^m$
			such that $\lambda_{l}\in K(z_{k_l})$ holds for all $l\in\N$ .
			Since we have $z_{k_l}\to\bar z$, we can deduce $\bar\lambda\in K(\bar z)$ from the closedness
			of $\gph K$, see \cref{lem:inner_semicompactness_of_K}. 
			Thus, the points from $\{(z_{k_l},\lambda_{l})\}_{l\in\N}$ as well as
			$(\bar z,\bar\lambda)$ are feasible to \eqref{eq:explicit_multipliers}.
			Since we have $f(z_{k_l})<f(\bar z)$ for all $l\in\N$, $(\bar z,\bar\lambda)$ cannot be a
			local minimizer of \eqref{eq:explicit_multipliers}.
			This, however, contradicts our assumptions.
\end{proof}

Naturally, when dealing with the reformulated problem
\eqref{eq:explicit_multipliers}, one can hardly expect to find
$\bar z$ such that $(\bar z,\lambda)$ is a local minimizer
{\em for all} $\lambda \in K(\bar z)$.
Typically, one simply finds a point $(\bar z,\bar\lambda)$
which is locally optimal (even this is often too optimistic
and one only gets a stationary point instead).
This underlines that the approach via problem
\eqref{eq:explicit_multipliers} is indeed quite far from ideal.
The proof of the previous theorem yields that
the situation can be saved if the stronger inner semicontinuity of $K$
is assumed. 

\begin{corollary}
	\label{cor:local_relationship}
	Let $(\bar z,\bar\lambda)\in\R^n\times\R^m$ be a local minimizer of
	\eqref{eq:explicit_multipliers} and
	let $K$ be inner semicontinuous at $(\bar z,\bar\lambda)$
	w.r.t.\ $\dom K$.
	Then $\bar z$ is a local minimizer of \eqref{eq:implicit_multipliers}.
\end{corollary}

Due to \cref{lem:inner_semicompactness_of_K}, the inner semicompactness of $K$ can be guaranteed via
local boundedness of $F$. In the context of bilevel programming, see \cref{sec:bilevel_programming},
we already mentioned that the mapping $F$ does not enjoy this property. 
However, $K$ is the Lagrange multiplier mapping associated with the lower level problem in this
context, and the latter is known to be locally bounded under validity of the
Mangasarian--Fromovitz constraint qualification which, thus, implies its inner semicompactness as well.
Hence, \cref{thm:local_relationship} recovers the popular result \cite[Theorem~3.2]{DempeDutta2012}.
However, exploiting \cref{prop:inner_semicompactness_of_multiplier_map}, we are in position to
weaken this constraint qualification such that the resulting multiplier mapping still possesses the
necessary inner semicompactness, i.e., we will keep the essential result of 
\cref{thm:local_relationship} in this case, 
see \cref{sec:app_bilevel_programming} for details.

\begin{remark}
 We point out that, in several exemplary cases, the inner semicompactness assumption will
 be automatically fulfilled (or will be a consequence of some initial assumption on the data,
 such as validity of a constraint qualification for the lower level program in bilevel programming).
 In fact, even a stronger form of inner semicompactness considered
 in \cref{sec:M_Stationarity_conditions} (i.e., inner semicompactness w.r.t.\ another
 intermediate mapping) will hold true, see \cref{sec:app_bilevel_programming} as well.
\end{remark}

Below, we illustrate that none of the assumptions which are necessary in order
to identify a local minimizer of \eqref{eq:implicit_multipliers} via local minimizers of
\eqref{eq:explicit_multipliers} can be dropped in \cref{thm:local_relationship}.
\begin{example}\label{ex:local_relationship}
\hfill
	\begin{enumerate}
	\item[(a)] Let us consider the setting $n=m=s=1$ with $f:=\id$,
		\[
			\forall z,\lambda\in\R\colon\quad
				F(z):=\begin{cases}
						\{0\}	&z>0,\\
						\{0,1\}	&z=0,\\
						\{1\}	&z<0,
					\end{cases}
				\qquad
				G(z,\lambda):=[-z-\lambda,\infty),
		\]
		as well as $M:=\R$. Obviously, we have $Z=[-1,\infty)$, i.e.,
		$\bar z:=-1$ is the uniquely determined local and global minimizer 
		of the associated problem \eqref{eq:implicit_multipliers}.
		It holds
		\[
			\forall z\in\R\colon\quad
				K(z)=	\begin{cases}
							\{0\} 	&z>0,\\
							\{0,1\}	&z=0,\\
							\{1\}	&z\in[-1,0),\\
							\varnothing	&z<-1.
						\end{cases}
		\]
		At $\tilde z:=0$, $K$ is inner semicompact w.r.t.\ $\dom K=Z$. It can be easily
		checked that $(\tilde z,0)$ is a local minimizer of \eqref{eq:explicit_multipliers}
		while $(\tilde z,1)$ is not. Thus, assertion (b) of \Cref{thm:local_relationship}
		does not generally hold whenever the local optimality for \eqref{eq:explicit_multipliers}
		cannot be guaranteed for all realizations of the implicit variables in question.
	\item[(b)] For $n=m=s=1$, we set $f:=\id$,
		\[
			\forall z,\lambda\in\R\colon\quad
			F(z):=
				\begin{cases}
					\{0\}		&	z\geq 0,\\
					\{-1/z\}	&	z<0,
				\end{cases}
			\qquad
			G(z,\lambda):=
				[-1,1+z],
		\]
		as well as $M:=\R$. Again, we find $Z=[-1,\infty)$, and $\bar z:=-1$ is the uniquely
		determined local and global minimizer of the associated problem \eqref{eq:implicit_multipliers}. 
		One easily computes
		\[
			\forall z\in\R\colon\quad
			K(z)=
				\begin{cases}
					\{0\}		&	z\geq 0,\\
					\{-1/z\}	&	z\in[-1,0),\\
					\varnothing	&	z<-1.
				\end{cases}
		\]
		Let us consider $\tilde z:=0$. At $\tilde z$, $K$ is not inner semicompact w.r.t.\ $\dom K$. 
		On the other hand, $(\tilde z,0)$ is a local minimizer
		of \eqref{eq:explicit_multipliers}
		since for $z\uparrow 0$, we have $\lambda\to\infty$ for any $\lambda\in K(z)$.
		Consequently, we cannot abstain from postulating inner semicompactness of $K$ in order
		to guarantee validity of assertion (b) of \Cref{thm:local_relationship}.
	\end{enumerate}	
\end{example}

The issues regarding the relationship between local minimizers of \eqref{eq:implicit_multipliers}
and \eqref{eq:explicit_multipliers} result from the fact that the variable $\lambda$ is equipped with a
metric structure in \eqref{eq:explicit_multipliers}. This is quite natural when treating
\eqref{eq:explicit_multipliers} as an optimization problem and forms the base for the derivation of
necessary optimality conditions, constraint qualifications, and solution algorithms. 
In the context of bilevel optimization, the authors of \cite{YeZhu2010} used a different approach by
disregarding the metric structure w.r.t.\ the lower level Lagrange multipliers in a reformulation
of the bilevel programming problem which exploits the optimal value function as well as the
Karush--Kuhn--Tucker-conditions of the lower level problem simultaneously, see
\cite[Proposition~3.1]{YeZhu2010}. In the more general setting considered here, this amounts to
saying that a feasible point $(\bar z,\bar\lambda)\in\R^n\times\R^m$ 
of \eqref{eq:explicit_multipliers} is a local minimizer 
whenever there is a neighborhood $U\subset\R^n$ of $\bar z$ such that
\begin{equation}\label{eq:local_minimizer_Q_metric_free}
	\forall (z,\lambda)\in \tilde Z\cap(U\times\R^m)\colon\quad
	f(\bar z)\leq f(z)
\end{equation}
holds. One can easily check that whenever 
$\bar z\in\R^n$ is a local minimizer of \eqref{eq:implicit_multipliers}, then, for each
$\lambda\in K(\bar z)$, $(\bar z,\lambda)$ is a local minimizer of \eqref{eq:explicit_multipliers}
in the sense of \eqref{eq:local_minimizer_Q_metric_free}. 
Conversely, if $(\bar z,\bar\lambda)\in\R^n\times\R^m$ is a local minimizer of
\eqref{eq:explicit_multipliers} in the sense of \eqref{eq:local_minimizer_Q_metric_free}, then
$\bar z$ is a local minimizer of \eqref{eq:implicit_multipliers}.
At the first glance, this approach seems to solve all the issues observed in 
\cref{thm:local_relationship}. However, it is completely unclear how local minimizers of
\eqref{eq:explicit_multipliers} in the sense of \eqref{eq:local_minimizer_Q_metric_free} can be
characterized or numerically identified. In \cite{YeZhu2010}, the authors do not comment on this
shortcoming at all. Thus, although faced with the disadvantages carved out in
\cref{thm:local_relationship}, we will stick to the classical notion of local optimality in
\eqref{eq:explicit_multipliers}, i.e., we impose a metric structure on $\lambda$ there.

We conclude this section with a brief remark regarding sufficient optimality conditions.
\begin{remark}\label{rem:implicit_multipliers_and_sufficient_optimality_conditions}
	In optimization theory, sufficient optimality conditions generally imply local
	linear or quadratic growth, i.e., they impose local isolatedness of the
	characterized minimizer.
	Now, consider a local minimizer $\bar z\in Z$ of \eqref{eq:implicit_multipliers}.
	Even if it is a locally isolated minimizer of \eqref{eq:implicit_multipliers},
	this property fails to hold for the respective local minimizer
	$(\bar z,\bar\lambda)$ for each $\bar\lambda\in K(\bar z)$ as soon as
	$K(\bar z)$ is a connected set which is not a singleton since the objective of
	\eqref{eq:explicit_multipliers} does not depend on the variable $\lambda$.
	As a consequence, one cannot rely on classical first- or second-order
	sufficient optimality conditions in order to characterize the minimizers
	of \eqref{eq:explicit_multipliers} as soon as $K$ is multivalued.
	In the light of \cref{thm:local_relationship}, it is, thus, difficult to 
	infer local optimality of a feasible point of \eqref{eq:implicit_multipliers}
	via \eqref{eq:explicit_multipliers}.
	In the context of bilevel and cardinality-constrained programming, this phenomenon has been mentioned 
	recently in \cite[Remark~4.21]{MehlitzZemkoho2019} 
	and \cite[Section~3.2]{BucherSchwartz2018}, respectively.
\end{remark}

\subsection{Necessary optimality conditions}\label{sec:M_Stationarity_conditions}

Here, we are going to investigate three notions of necessary optimality conditions 
of M-stationarity-type for program \eqref{eq:implicit_multipliers}.
Two of them correspond to (standard) M-stationarity defined in \cref{Def:Mstat}
applied either directly to \eqref{eq:implicit_multipliers} or to its reformulation
\eqref{eq:explicit_multipliers}, while the third notion offers conditions in a completely explicit form.
Without mentioning it again, we assume throughout the section that $\gph H$ is (locally) closed.
As we have seen in \cref{lem:H_closed_graph} this is inherent if $F$ is locally bounded, but
it also holds in far more general but practically relevant situations.
In particular, $\gph H$ is locally closed around $(\bar z,0) \in \gph H$
if the mapping $\widehat{K}$ from \eqref{eq:simple_intermediate_map} below is 
inner semicompact at $(\bar z,0)$ w.r.t.\ $\dom \widehat{K} = \gph H$.
For later use, let us define set-valued mappings $\mathcal H\colon\R^n\times\R^m\tto\R^{m}\times\R^s$ and
$\mathfrak{H}\colon\R^n\times\R^m\rightrightarrows\R^{n+m}\times\R^{n+m+s}$
by means of
\begin{equation*}
	\begin{aligned}
		\forall z\in\R^n,\,\forall\lambda\in\R^m\colon\quad
		\mathcal{H}(z,\lambda) &:=
			\begin{pmatrix}
				F(z)-\lambda\\
				G(z,\lambda)
			\end{pmatrix},
		\;\;
		\mathfrak{H}(z,\lambda):=
			\begin{pmatrix} 
				(z,\lambda)-\gph F\\
				(z,\lambda,0)-\gph G
			\end{pmatrix}.
	\end{aligned}
\end{equation*}
In order to incorporate the constraints $z \in M$,
we will exploit the mappings
$H_M\colon\R^n\tto\R^s\times\R^n$, 
$\mathcal H_M\colon\R^n\times\R^m\tto\R^{m}\times\R^s\times\R^n$, as well as
$\mathfrak{H}_M\colon\R^n\times\R^m\rightrightarrows\R^{n+m}\times\R^{n+m+s}\times\R^n$ obtained
from $H$, $\mathcal H$, and $\mathfrak H$
by adding $(z - M)$ as a component, 
i.e., we exemplary have $H_M(z) := H(z) \times (z - M)$ for all $z\in\R^n$.
Observe that we have $z\in Z$ if and only if $(0,0)\in H_M(z)$ and $(z,\lambda)\in\tilde Z$
if and only if $(0,0,0)\in\mathcal H_M(z,\lambda)$ if and only if $((0,0),(0,0,0),0)\in\mathfrak H_M(z,\lambda)$.
\begin{definition}\label{def:Mst_explicit_variable}
	Let $\bar z\in\R^n$ be a feasible point of \eqref{eq:implicit_multipliers}.
	Then $\bar z$ is said to be
	\begin{enumerate}
	\item[(i)]
	\emph{implicitly M-stationary}
	if there exists $\nu\in\R^s$ such that
	\[
		0\in\partial f(\bar z)+D^*H(\bar z,0)(\nu)+\mathcal N_M(\bar z),
	\]
	\item[(ii)]
	\emph{fuzzily M-stationary} (or \emph{fuzzily M-stationary w.r.t.\ $\bar\lambda$})
	if there exist $\bar\lambda\in K(\bar z)$, $\mu\in\R^m$, and $\nu\in\R^s$
	such that
	\[
		(0,0) \in \partial f(\bar z)\times\{0\} +D^* \mathcal{H}((\bar z,\bar\lambda),(0,0))(\mu,\nu)
                        +\mathcal N_M(\bar z)\times\{0\}
	\]
	holds, i.e., if $(\bar z,\bar\lambda)$ is M-stationary for \eqref{eq:explicit_multipliers},
	\item[(iii)]
	\emph{explicitly M-stationary} (or \emph{explicitly M-stationary w.r.t.\ $\bar\lambda$})
	if there exist $\bar\lambda\in K(\bar z)$, $\mu\in\R^m$, and $\nu\in\R^s$
	such that
	\[
		0\in\partial f(\bar z)+D^*F(\bar z,\bar\lambda)(\mu)
			+\{\xi\in\R^n\,|\,(\xi,\mu)\in D^*G((\bar z,\bar\lambda),0)(\nu)\}
			+\mathcal N_M(\bar z).
	\]
	\end{enumerate}
\end{definition}

\cref{Pro:MSRMain} now immediately provides
constraint qualifications for the above M-stationarity notions.

\begin{proposition}\label{prop:Stationarities_via_subreguarity}
 	Let $\bar z\in \R^n$ be a local minimizer of \eqref{eq:implicit_multipliers}.
 	Then the following assertions hold.
 	\begin{enumerate}
  		\item[(a)] If $H_M$ is metrically subregular at $(\bar z,(0,0))$, 
  			then $\bar z$ is implicitly M-stationary.
  		\item[(b)] If $\mathcal{H}_M$ is metrically subregular at 
  			$((\bar z,\bar\lambda),(0,0,0))$ for some $\bar\lambda\in K(\bar z)$, 
  			then $\bar z$ is fuzzily M-stationary w.r.t.\ $\bar\lambda$.
  		\item[(c)] If $\mathfrak{H}_M$ is metrically subregular at 
  			$((\bar z,\bar\lambda),((0,0),(0,0,0),0))$ 
  			for some $\bar\lambda\in K(\bar z)$, then $\bar z$ is explicitly M-stationary
  			w.r.t.\ $\bar\lambda$.
 	\end{enumerate}
\end{proposition}
\begin{proof}
	Since $\bar z$ is a local minimizer of \eqref{eq:implicit_multipliers}, \cite[Theorem~6.1]{Mordukhovich2018}
	guarantees validity of
	\[
		0\in\partial f(\bar z)+\mathcal N_{H_M^{-1}(0,0)}(\bar z).
	\]
	Noting that for each $\lambda\in K(\bar z)$, $(\bar z,\lambda)$ is a local minimizer of
	\eqref{eq:explicit_multipliers}, see \cref{thm:local_relationship}, 
	\cite[Theorem~6.1]{Mordukhovich2018} furthermore implies
	\begin{subequations}\label{eq:abstract_M_stationarity}
		\begin{align}
			\label{eq:abstract_M_stationarity_fuzzy}
				(0,0)&\in\partial f(\bar z)\times\{0\}
					+\mathcal N_{\mathcal H_M^{-1}(0,0,0)}(\bar z,\bar\lambda),\\
			\label{eq:abstract_M_stationarity_explicit}
				(0,0)&\in\partial f(\bar z)\times\{0\}
					+\mathcal N_{\mathfrak H_M^{-1}((0,0),(0,0,0),0)}(\bar z,\bar\lambda).
		\end{align}
	\end{subequations}
	For the proof of (a), we exploit the metric subregularity of $H_M$ at $(\bar z,(0,0))$ 
	and \cref{Pro:MSRMain} in order to
	find $0\in\partial f(\bar z)+D^*H_M(\bar z,(0,0))(\nu,\xi)$ for some $\nu\in\R^s$ and $\xi\in\R^n$.
	Now, we can exploit assertion (a) of \cref{lem:equality_coderivative_products} 
	in order to find
	\[
		D^*H_M(\bar z,(0,0))(\nu,\xi)
		=
		\begin{cases}
			D^*H(\bar z,0)(\nu) + \xi 		&	\xi\in\mathcal N_M(\bar z),\\
			\varnothing				&	\text{otherwise,}
		\end{cases}
	\]
	and this shows that $\bar z$ is implicitly M-stationary.\\
	The proof for (b) works in analogous way exploiting \eqref{eq:abstract_M_stationarity_fuzzy}.\\
	In order to verify statement (c), we first introduce a continuously 
	differentiable single-valued mapping
	$\mathfrak h_M\colon\R^n\times\R^m\to\R^{n+m}\times\R^{n+m+s}\times\R^n$
	and a set $\Omega\subset\R^{n+m}\times\R^{n+m+s}\times\R^n$ by means of
	\[
		\forall z\in\R^n,\,\forall\lambda\in\R^m\colon\quad
		\mathfrak h_M(z,\lambda)
		:=
		((z,\lambda),(z,\lambda,0),z),
		\qquad
		\Omega
		:=
		\gph F\times\gph G\times M
	\]
	and observe that $\mathfrak H_M(z,\lambda)=\mathfrak h_M(z,\lambda)-\Omega$
	holds for all $z\in\R^n$ and $\lambda\in\R^m$. 
	Thus, the metric subregularity of the feasibility mapping $\mathfrak H_M$ at
	the reference point guarantees applicability of the pre-image rule,
	see \cref{sec:abstract_M_St}.
	Together with the product rule for the limiting normal cone, see
	\cite[Proposition~6.41]{RockafellarWets1998}, we obtain
	\begin{align*}
		&\mathcal N_{\mathfrak H_M^{-1}((0,0),(0,0,0),0)}(\bar z,\bar\lambda)\\
		&\qquad
		\subset
		\nabla\mathfrak h_M(\bar z,\bar\lambda)^\top
			\mathcal N_\Omega(\mathfrak h_M(\bar z,\bar\lambda))\\
		&\qquad
		=
		\nabla\mathfrak h_M(\bar z,\bar\lambda)^\top
			\left(
				\mathcal N_{\gph F}(\bar z,\bar\lambda)
				\times
				\mathcal N_{\gph G}(\bar z,\bar\lambda,0)
				\times
				\mathcal N_M(\bar z)
			\right)\\
		&\qquad
		=\left\{
			(\xi_1+\xi_2+\xi_3,\mu_1+\mu_2)
			\middle|\,
				\begin{aligned}
					&\exists \nu\in\R^s\colon\,\xi_1\in D^*F(\bar z,\bar\lambda)(-\mu_1),\\
					&\qquad (\xi_2,\mu_2)\in D^*G((\bar z,\bar\lambda),0)(\nu),\,
					\xi_3\in\mathcal N_M(\bar z)
				\end{aligned}
		 \right\}.
	\end{align*}
	Now, the claim follows from \eqref{eq:abstract_M_stationarity_explicit}.
\end{proof}

\begin{remark}\label{rem:handling_of_abstract_constraints}
	Taking into account \cref{lem:metric_subregularity_of_productsII}
	as well as the fact that the simple mapping
	$z\tto z-M$ is trivially metrically subregular 
	at all points of its graph,
 	metric subregularity of $H_M$ at $(\bar z,(0,0))$ can be ensured by
 	metric subregularity of $H$ at $(\bar z,0)$ and metric subregularity
 	of $z \tto (z - H^{-1}(0)) \times (z - M)$ 
 	at $(\bar z,(0,0))$.	
 	Note that the latter implies applicability of the intersection rule
 	for the set $Z=H_M^{-1}(0,0) = H^{-1}(0) \cap M$, namely
 	\begin{equation*}
  		\mathcal N_{H_M^{-1}(0,0)}(\bar z) 
  		\subset 
  		\mathcal N_{H^{-1}(0)}(\bar z) + \mathcal N_{M}(\bar z)
 	\end{equation*}
 	and, hence, combined with the metric subregularity of $H$, also provides a constraint
 	qualification for implicit M-stationarity, see \cref{Pro:MSRMain}.
\end{remark}

Let us briefly comment on how to ensure the above metric subregularity assumptions
via sufficient conditions in terms of coderivatives.
Let $\bar z\in \R^n$ be a feasible point for \eqref{eq:implicit_multipliers}
and let $\bar \lambda \in K(\bar z)$ be chosen arbitrarily.
\begin{enumerate}
 	\item[(i)] If, for all $\nu\in\R^s$ and $\xi\in\R^n$, the implication
 		\begin{equation*}
  			- \xi \in D^*H(\bar z,0)(\nu) \cap \big(- \mathcal N_M(\bar z)\big)
  			\ \Longrightarrow \ \xi,\nu = 0
 		\end{equation*}
 		holds, then $H_M$ is metrically subregular at $(\bar z,(0,0))$.
 	\item[(ii)] If, for all $\mu\in\R^m$, $\nu\in\R^s$, and $\xi\in\R^n$, the implication
 		\begin{equation*}
 		 	(-\xi,0) \in D^*\mathcal{H}((\bar z,\bar \lambda),(0,0))(\mu,\nu),
  			\ \xi \in \mathcal N_M(\bar z)
  			\ \Longrightarrow \ \xi,\mu,\nu = 0
 		\end{equation*}
 		holds, then $\mathcal{H}_M$ is metrically subregular at 
 		$((\bar z,\bar \lambda),(0,0,0))$.
 	\item[(iii)] If, for all $\zeta \in\R^n$, $\mu\in\R^m$, $\nu\in\R^s$, and $\xi\in\R^n$,
 		the implication
 		\begin{equation*}
			\zeta \in D^*F(\bar z,\bar\lambda)(\mu), \
 			(-\zeta - \xi,\mu) \in D^*G((\bar z,\bar \lambda),0)(\nu),
 			\ \xi \in \mathcal N_M(\bar z) \ \Longrightarrow \
 			\xi,\zeta,\mu,\nu = 0
 		\end{equation*}
 		holds, then $\mathfrak{H}_M$ is metrically subregular at 
 		$((\bar z,\bar\lambda),((0,0),(0,0,0),0))$.
\end{enumerate}
Observe that the above conditions correspond to the famous Mordukhovich criterion
applied to the respective situation at hand. This can be seen from
statement (a) of \cref{lem:equality_coderivative_products}, the sum rule for
coderivative calculus, see e.g.\ \cite[Theorem~3.9]{Mordukhovich2018}, and
the product rule for limiting normals, see \cite[Proposition~6.41]{RockafellarWets1998}.
Thus, the above conditions already imply metric regularity of the mappings $H_M$,
$\mathcal H_M$, and $\mathfrak H_M$ at the respective point of interest.
Weaker sufficient conditions for the presence of metric subregularity 
in terms of limiting normal cones and coderivatives can exemplary be found in 
\cite{HenrionJouraniOutrata2002,HenrionOutrata2005,IofOut08,ZheNg10}.
Even finer sufficient conditions can be obtained
using the directional limiting approach, namely the
\emph{first-order sufficient condition for metric subregularity}
from \cite{GfrererKlatte2016} or the directional pseudo- and quasi-normality 
conditions from \cite{BaiYeZhang2019,BenkoCervinkaHoheisel2019}.
For more details on the directional limiting approach in variational analysis, we 
exemplary refer to \cite{BenkoGfrererOutrata2019,GfrererOutrata2016b}
and the references therein.

Next, we are going to compare the three approaches on how to come up with necessary
optimality conditions for \eqref{eq:implicit_multipliers}. 
First, we look at the qualification conditions
and, afterwards, we deal with the stationarity conditions.
For that purpose, let us define another intermediate
mapping $\widehat K\colon\R^n\times\R^s \tto \R^m$,
closely related to $K$, given by
        	\begin{equation}\label{eq:simple_intermediate_map}
        		\forall z\in\R^n,\,\forall w\in\R^s\colon\quad
				\widehat K(z,w):= \{\lambda \in F(z) \,|\, w \in G(z,\lambda)\}.
        	\end{equation}
\begin{proposition}\label{prop:MSR_of_H_via_cal_H}
	Let $\bar z \in\R^n$ be a feasible point of \eqref{eq:implicit_multipliers}.
 	Consider the following two assumptions:
	 \begin{enumerate}
  		\item [(a)] $\mathcal{H}_M$ is metrically subregular at $((\bar z,\lambda),(0,0,0))$
 			for each $\lambda\in K(\bar z)$ and $\widehat K$
 			is inner semicompact at $(\bar z,0)$ w.r.t.\ $\dom \widehat K$,
 		\item [(b)] $\mathcal{H}_M$ is metrically subregular at $((\bar z,\bar\lambda),(0,0,0))$
 			for some $\bar\lambda\in K(\bar z)$ and $\widehat K$
 			is inner semicontinuous at $((\bar z,0),\bar\lambda)$ w.r.t.\ $\dom \widehat K$.
 	\end{enumerate}
 	Then each of the conditions (a) and (b) implies that 
 	$H_M$ is metrically subregular at $(\bar z,(0,0))$.
\end{proposition}
\begin{proof}
	Let us first show the statement under validity of (a).
	Suppose that $H_M$ is not metrically subregular at $(\bar z,(0,0))$.
	Hence, for each $k\in\N$, we find $z_k\in\R^n$ 
	and some $w_k\in H(z_k)$ satisfying
	\[
		\dist(z_k,Z)
		=
		\dist(z_k,H^{-1}_M(0,0))
		>
		k\,\dist((0,0),H_M(z_k)) 
		=
		k(\norm{w_k}{}+\dist(z_k,M))
	\]
	such that $z_k\to\bar z$. 
	Particularly, $w_k\to 0$ follows.
	By definition of $H$ and $\widehat{K}$, we have $(z_k,w_k) \in \dom \widehat K$ for each $k\in\N$.
	The assumed inner semicompactness of $\widehat K$ at $(\bar z,0)$ w.r.t.\ $\dom\widehat K$
	yields the existence of 
	a sequence $\{\lambda_k\}_{k\in\N}$ and some $\lambda\in \R^m$ such that $\lambda_k \to\lambda$ 
	and $\lambda_k\in\widehat K(z_k,w_k)$ hold 	along a subsequence (without relabeling).
	Recalling that $F$ and $G$ are mappings with closed graphs, 
	the above convergences yield 
	$\lambda\in F(\bar z)$ and $0\in G(\bar z,\lambda)$,
	i.e., $\lambda\in\widehat K(\bar z,0)=K(\bar z)$.
	Consequently, we obtain
	\begin{align*}
                &\dist((z_k,\lambda_k),\mathcal{H}^{-1}_M(0,0,0))
                =
                \dist((z_k,\lambda_k),\tilde Z)
                \geq
                \dist(z_k,Z)\\
                &\qquad
                > 
                k\bigl(\norm{w_k}{}+\dist(z_k,M)\bigr)
                \geq 
                k \bigl(\dist(0,G(z_k,\lambda_k))+\dist(z_k,M)\bigr)\\
                &\qquad
                =
                k\bigl(\dist(0,F(z_k)-\lambda_k)+\dist(0,G(z_k,\lambda_k))+\dist(z_k,M)\bigr)\\
                &\qquad
                =
                k\dist((0,0,0),\mathcal{H}_M(z_k,\lambda_k)),
	\end{align*}
	showing that $\mathcal{H}_M$ is not metrically subregular at
	$((\bar z,\lambda),(0,0,0))$.
	This, however, contradicts the proposition's assumptions from (a).\\
	The proof works similarly under validity of (b) observing that the sequence $\{\lambda_k\}_{k\in\N}$
	can be chosen to be convergent to the fixed implicit variable 
	$\bar\lambda\in \widehat{K}(\bar z,0)$ by
	inner semicontinuity of $\widehat{K}$ w.r.t.\ $\dom\widehat K$ at $((\bar z,0),\bar\lambda)$.
\end{proof}

\begin{proposition}\label{prop:MSR_of_cal_H_via_frak_H}
	Let $\bar z\in\R^n$ be a feasible point of \eqref{eq:implicit_multipliers} and fix $\bar\lambda\in K(\bar z)$.
	Assume that $\mathfrak{H}_M$ is
	metrically subregular at $((\bar z,\bar\lambda),((0,0),(0,0,0),0))$.
	Then $\mathcal{H}_M$ is metrically
	subregular at $((\bar z,\bar\lambda),(0,0,0))$.
\end{proposition}
\begin{proof}
    The proof follows easily from
    $\mathcal{H}^{-1}_M(0,0,0)=\mathfrak{H}^{-1}_M((0,0),(0,0,0),0)=\tilde Z$,
    together with the simple estimates
    \[
    	\begin{aligned}
     		 &\forall\, (z,\lambda')\in \gph F,\,\forall\lambda\in\R^m\colon\quad&
     		 \dist((z,\lambda),\gph F)  &\leq \norm{\lambda - \lambda'}{},&
     			\\
     		&\forall\,((z,\lambda),w)\in\gph G\colon\quad&
     		\dist(((z,\lambda),0),\gph G) &\leq \norm{w}{} & 	     			
   		\end{aligned}
   	\]
    which yield that
    $\dist(((0,0),(0,0,0),0),\mathfrak{H}_M(z,\lambda)) \leq
     \dist((0,0,0),\mathcal{H}_M(z,\lambda))$
    holds.
\end{proof}

Next, we investigate the relationship between the three stationarity
notions from \cref{def:Mst_explicit_variable} in more detail.
Let us first address the relationship between fuzzy 
and explicit M-stationarity.
\begin{remark}\label{rem:fuzzy_vs_explicit_M_stationarity}
	\Cref{lem:product_rule,lem:coderivative_product,lem:equality_coderivative_products} yield
	that, under suitable conditions,
	the coderivative of $\mathcal{H}$ at some point $((\bar z,\lambda),(0,0))$
	of its graph can be estimated or computed via its components, i.e.,
	\begin{equation}\label{eq:coder_CalH}\tag{Inc$(\lambda)$}
    	D^* \mathcal{H}((\bar z,\lambda),(0,0))(\mu,\nu) \subset
    	D^* F(\bar z,\lambda)(\mu)\times\{-\mu\} \, + \,
    	D^* G((\bar z,\lambda),0)(\nu)
	\end{equation}
	for all $\mu\in\R^m$ and $\nu\in\R^s$.
	A sufficient condition for validity of \eqref{eq:coder_CalH} is given by
	\[
		\xi\in D^*F(\bar z,\lambda)(0),\,(-\xi,0)\in D^*G((\bar z,\lambda),0)(0)
		\quad
		\Longrightarrow
		\quad
		\xi=0,
	\]
	see \cref{lem:product_rule},
	and this is inherent if either $F$ possesses the Aubin property at $(\bar z,\lambda)$
	or $G$ possesses the Aubin property at $((\bar z,\lambda),0)$.
	Another situation where \eqref{eq:coder_CalH} naturally holds is given in the case where
	$F$ and $G$ are polyhedral set-valued mappings.
	Note that \eqref{eq:coder_CalH} does not hold for free in general
	due to fact that the variable $z$
	enters both set-valued parts $F$ and $G$.
	Combining validity of \eqref{eq:coder_CalH} with simple computations 
	provides conditions that guarantee that fuzzy M-stationarity implies 
	or even coincides with explicit M-stationarity.
	We will show in
	\cref{sec:optimality_conditions_via_special_structure} 
	that the latter is automatically fulfilled
	for all of our example problems from \cref{sec:examples},
	so we skip further details and consider these two notions
	to be basically identical.
\end{remark}

Let us now tackle the more interesting question when an implicitly 
M-stationary point is fuzzily or explicitly M-stationary as well.
The latter question can be easily answered by the chain rule
since $H(z)=(G \circ \tilde F)(z)$ holds for all $z\in\R^n$ where
$\tilde F\colon\R^n \tto \R^n\times\R^m$ is the set-valued mapping given by
$\tilde F(z):= \{z\}\times F(z)$ for all $z\in\R^n$.
Taking into account that
\begin{equation*}\label{eq:tilde_F_coder}
 D^*\tilde F(z,(z,\lambda))(\xi,\mu) = \xi + D^*F(z,\lambda)(\mu)    
\end{equation*}
holds by \cref{lem:equality_coderivative_products} (a),
\cref{lem:chain_rule} immediately yields the following result.
\begin{proposition}\label{prop:from_implicit_to_explicit_M_St}
	Let $\bar z\in\R^n$ be an implicitly M-stationary point for \eqref{eq:implicit_multipliers}.
	Then $\bar z$ is explicitly M-stationary provided
	the mapping
	\begin{equation}\label{eq:New_aux_map}
	  (z,w) \tto \tilde F(z) \cap G^{-1}(w)
	  =
	  \{(z,\lambda)\,|\,\lambda\in F(z),\,w\in G(z,\lambda)\}
	\end{equation}
	is inner semicompact at $(\bar z,0)$ w.r.t.\ its domain while the mapping
  	\begin{equation}\label{eq:New_feas_map}
  	(z,w,q,\lambda) \tto 
  	\bigl(z - q,F(z) - \lambda,G(q,\lambda) - w\bigr)
  	\end{equation}
  	is metrically subregular at 
  	$((\bar z,0,\bar z,\lambda),(0,0,0))$
  	for each $\lambda\in K(\bar z)$.
\end{proposition}

On the other hand, using the intermediate mapping $\widehat{K}$ from \eqref{eq:simple_intermediate_map},
one obtains
\begin{equation*}
 \gph H=\dom \widehat K, \qquad\qquad
 \gph \widehat K = \widehat{\mathcal{H}}^{-1}(0,0)
\end{equation*}
for the feasibility mapping
$\widehat{\mathcal{H}}\colon\R^n\times\R^s\times\R^m\tto \R^m\times\R^s$ given by
			\begin{equation}\label{eq:simple_feasible_map}
			\forall z\in\R^n,\,\forall w\in\R^s,\,\forall \lambda\in\R^m\colon\quad
 			\widehat{\mathcal{H}}(z,w,\lambda) := 
 				\begin{pmatrix}
        			F(z) - \lambda \\
        			G(z,\lambda) - w
        		\end{pmatrix}
        		= \mathcal{H}(z,\lambda) -
        		\begin{pmatrix}
        			0 \\
        			w
        		\end{pmatrix}.
			\end{equation}
Keeping \cref{lem:chain_rule} and the subsequently stated remarks in mind,
the coderivative of $H$ can be estimated via the coderivative of $\widehat{\mathcal{H}}$ under suitable assumptions.
More precisely, for $(a,b)$ satisfying $a\in D^*H(\bar z,0)(b)$,
 the above approach yields
 the existence of $\bar\lambda\in K(\bar z)$
 together with $\mu\in\R^m$ and $\nu\in\R^s$ such that
 \begin{equation*}
    (a,-b,0) \in D^*\widehat{\mathcal{H}}((\bar z,0,\bar\lambda),(0,0))(\mu,\nu).
 \end{equation*}
 Consequently, the decoupled sum rule from \cite[Section~5.4]{BenkoMehlitz2020}
 implies that the relations $\nu=b$ and $(a,0) \in D^*\mathcal{H}((\bar z,\bar\lambda),(0,0))(\mu,\nu)$
 need to hold, and for the particular choice
 $a \in -\partial f(\bar z) - \mathcal N_M(\bar z)$,
 we end up with fuzzy M-stationarity.

We summarize these observations in the subsequent proposition.
 
\begin{proposition}\label{prop:from_implicit_to_explicit_M_stationarity_1}
	Let $\bar z\in\R^n$ be an implicitly M-stationary point of \eqref{eq:implicit_multipliers}.
	Assume that $\widehat K$ is inner semicompact at $(\bar z,0)$ w.r.t.\ $\dom\widehat K$,
	and	let $\widehat{\mathcal{H}}$ be metrically subregular at $((\bar z,0,\lambda),(0,0))$ for each
	$\lambda\in K(\bar z)$. Then $\bar z$ is fuzzily M-stationary.
	Moreover, $\bar z$ is also explicitly M-stationary if \eqref{eq:coder_CalH} holds
	for each $\lambda\in K(\bar z)$.
\end{proposition}

The metric subregularity assumption in
\cref{prop:from_implicit_to_explicit_M_stationarity_1}
can again be replaced by the stronger Mordukhovich criterion.
\begin{corollary}\label{cor:from_implicit_to_explicit_M_St}
	Let $\bar z\in\R^n$ be an implicitly M-stationary point for \eqref{eq:implicit_multipliers}.
	Let $\widehat K$ be inner semicompact at $(\bar z,0)$ w.r.t.\ $\dom\widehat K$.
	\begin{enumerate}
	\item[(a)]
		Assume that, for each $\lambda\in K(\bar z)$, the constraint qualification
		\begin{equation}\label{eq:abstract_CQ}
		(0,0) \in D^*\mathcal{H}((\bar z,\lambda),(0,0))(\mu,0)
		\quad \Longrightarrow \quad 
		\mu = 0
		\end{equation}
		holds, which is inherent whenever $\mathcal{H}$
		is metrically regular at $((\bar z,\lambda),(0,0))$.
   		Then $\bar z$ is fuzzily M-stationary.
   		If, additionally, \eqref{eq:coder_CalH} holds for each 
   		$\lambda\in K(\bar z)$, then $\bar z$ is explicitly M-stationary.
   	\item[(b)] Assume that \eqref{eq:coder_CalH} holds for each $\lambda\in K(\bar z)$
   		and let	
		\begin{equation}\label{eq:metric_regularity_CQ}
		\xi \in D^*F(\bar z,\lambda)(\mu),\,
		(-\xi,\mu) \in D^*G((\bar z,\lambda),0)(0)
		\quad\Longrightarrow\quad 
		\mu = 0
		\end{equation}
		be valid for each $\lambda\in K(\bar z)$
		which is inherent whenever $G$ posseses the Aubin property at $((\bar z,\lambda),0)$
		for each $\lambda\in K(\bar z)$.
		Then $\bar z$ is explicitly M-stationary.
	\item[(c)] Assume that
		\begin{equation}\label{eq:strong_CQ}
		\xi \in D^*F(\bar z,\lambda)(\mu),\,
		(-\xi,\mu) \in D^*G((\bar z,\lambda),0)(0)
		\quad\Longrightarrow\quad 
		\xi,\mu = 0
		\end{equation}
		is valid for each $\lambda\in K(\bar z)$
		which is inherent whenever $G$ posseses the Aubin property at $((\bar z,\lambda),0)$
		for each $\lambda\in K(\bar z)$.
		Then $\bar z$ is explicitly M-stationary.
	\end{enumerate}
\end{corollary}
\begin{proof}
	Let us start with the proof of statement (a).
	The relation between the coderivatives of $\mathcal H$ and $\widehat{\mathcal H}$,
	which has been discussed above
    \cref{prop:from_implicit_to_explicit_M_stationarity_1},
	yields 
	\[
		\ker D^*\widehat{\mathcal H}((\bar z,0,\lambda),(0,0))
		=
		\bigl\{(\mu,0)\,\big|\,(\mu,0)\in\ker D^*\mathcal H((\bar z,\lambda),(0,0))\bigr\}
	\]	
	for each $\lambda\in K(\bar z)$.
	Thus, the constraint qualification \eqref{eq:abstract_CQ} implies
	metric regularity of $\widehat{\mathcal{H}}$ at the point
	$((\bar z,0,\lambda),(0,0))$ for each $\lambda\in K(\bar z)$ by
	means of the Mordukhovich criterion.
	Thus, the assertion follows from \cref{prop:from_implicit_to_explicit_M_stationarity_1}.
	For the proof of (b), we observe that due to validity of \eqref{eq:coder_CalH} for each 
	$\lambda\in K(\bar z)$, \eqref{eq:metric_regularity_CQ} implies \eqref{eq:abstract_CQ}
	for each $\lambda\in K(\bar z)$, i.e., the assertion follows from statement (a).
	Finally, (c) follows from (b) observing that
	\eqref{eq:strong_CQ} implies validity of \eqref{eq:metric_regularity_CQ}
	as well as \eqref{eq:coder_CalH} by means of \cref{rem:fuzzy_vs_explicit_M_stationarity}.
\end{proof}

Let us mention that, in general, explicitly M-stationary points do not need to be implicitly M-stationary,
see e.g.\ \cref{sec:app_CCMP} and \cite[Section~4]{Mehlitz2019b} where this issue is visualized 
in the context of cardinality- and or-constrained programming, respectively.

As a consequence of \cref{prop:Stationarities_via_subreguarity,prop:from_implicit_to_explicit_M_stationarity_1,prop:from_implicit_to_explicit_M_St}, 
we obtain the following constraint qualifications which guarantee that
a given local minimizer of \eqref{eq:implicit_multipliers} is explicitly M-stationary.
\begin{theorem}\label{thm:CQs_for_M_stationarities}
	Let $\bar z\in \R^n$ be a local minimizer of \eqref{eq:implicit_multipliers}
 	and consider the following assumptions:
 	\begin{enumerate}
 		\item[(a)] $H_M$ is metrically subregular at $(\bar z,(0,0))$,
  			$\widehat K$ is inner semicompact at $(\bar z,0)$ w.r.t.\ $\dom \widehat K$,
  			and for each $\lambda\in K(\bar z)$,
  			$\widehat{\mathcal{H}}$ is metrically subregular at $((\bar z,0,\lambda),(0,0))$ 
  			while the inclusion \eqref{eq:coder_CalH} holds,
  		\item[(b)] $H_M$ is metrically subregular at $(\bar z,(0,0))$,
  			the mapping defined in \eqref{eq:New_aux_map} 
  			is inner semicompact at $(\bar z,0)$ w.r.t.\ its domain,
  			and for each $\lambda\in K(\bar z)$, the mapping defined in \eqref{eq:New_feas_map} is 
  			metrically subregular at 
  			$((\bar z,0,\bar z,\lambda),(0,0,0))$,
  		\item[(c)] $\mathcal{H}_M$ is metrically subregular at $((\bar z,\bar\lambda),(0,0,0))$
  			for some $\bar\lambda\in K(\bar z)$ and the inclusion 
  			\textup{\hyperref[eq:coder_CalH]{(Inc$(\bar\lambda)$)}} holds,
  		\item[(d)] $\mathfrak H_M$ is metrically subregular at 
  			$((\bar z,\bar \lambda),((0,0),(0,0,0),0))$
  			for some $\bar\lambda\in K(\bar z)$.
 	\end{enumerate}
 	Then each of (a), (b), (c), and (d) implies that $\bar z$ is explicitly M-stationary.
\end{theorem}

Once more, let us emphasize that validity of the inclusion
\eqref{eq:coder_CalH} can be guaranteed under not too restrictive conditions
which can be found in \cref{lem:coderivative_product,lem:equality_coderivative_products},
see \cref{rem:fuzzy_vs_explicit_M_stationarity} as well.

In \cref{sec:optimality_conditions_via_special_structure}, 
we will see that the most essential
ingredients from \cref{thm:CQs_for_M_stationarities} are the metric subregularity
assumptions on $H_M$, $\mathcal H_M$, and $\mathfrak H_M$, respectively,
while the other requirements can often be guaranteed by simpler structure of the problem
or by suitable qualification conditions.

A general purpose of necessary optimality conditions is to shrink the feasible set
down to a (hopefully) small number of points which are potential candidates for
local minimizers. Since we are interested in finding local minimizers $\bar z\in \R^n$ of 
\eqref{eq:implicit_multipliers} even if \eqref{eq:explicit_multipliers} is under
consideration, \cref{thm:local_relationship} underlines that validity of
explicit M-stationarity w.r.t.\ \emph{each} $\bar\lambda\in K(\bar z)$
is desirable for that purpose since this corresponds to standard
M-stationarity for \eqref{eq:explicit_multipliers} at its local minimizers $(\bar z,\bar\lambda)$
under mild assumptions, see \cref{rem:fuzzy_vs_explicit_M_stationarity} again.
In order to infer this directly from \eqref{eq:explicit_multipliers}, one
has to impose metric subregularity of $\mathfrak H_M$ at all points
$((\bar z,\bar \lambda),((0,0),(0,0,0),0))$ such that $\bar\lambda\in K(\bar z)$ holds.
Postulating that the intermediate mapping $\widehat K$ from 
\eqref{eq:simple_intermediate_map} is inner semicompact w.r.t.\ $\dom\widehat K$ at
$(\bar z,0)$ (which is inherent in several underlying applications), 
\cref{prop:MSR_of_H_via_cal_H,prop:MSR_of_cal_H_via_frak_H} guarantee that
$H_M$ is metrically subregular at $(\bar z,(0,0))$. 
Due to \cref{prop:Stationarities_via_subreguarity}, this is already enough
to ensure that the local minimizer $\bar z$ of \eqref{eq:implicit_multipliers}
is implicitly M-stationary.
In this regard, the constraint qualifications
one has to impose on \eqref{eq:explicit_multipliers} in order to find potential
candidates for local minimizers of \eqref{eq:implicit_multipliers} via the
associated M-stationarity conditions are in several situations not weaker
than the constraint qualifications one has to impose directly on \eqref{eq:implicit_multipliers}
in order to infer implicit M-stationarity as a necessary optimality condition.
In the context of bilevel optimization, the counterexamples in \cite[Section~3.1]{AdamHenrionOutrata2018}
depict that the converse statement of \cref{prop:MSR_of_H_via_cal_H} does not hold
in general. This also shows that, in some situations, the conditions in 
statement (a) of \cref{thm:CQs_for_M_stationarities} might be weaker than the 
condition from statement (d) (but for all $\bar\lambda\in K(\bar z)$) in this regard. 
Let us sum up some important points:
\begin{enumerate}
	\item[(i)] it may happen that $\bar z$ is implicitly M-stationary while it is \emph{not}
		explicitly M-stationary w.r.t.\ \emph{all} elements of $K(\bar z)$
		(in case where the chain rule is applicable in order to compute the
		coderivative of $H$, one can only ensure explicit M-stationarity w.r.t.\
		those implicit variables which are \emph{active} in the union appearing in the
		chain rule), however, under mild assumptions, explicit M-stationarity 
		for \emph{some} instance of the implicit variable can be derived,
	\item[(ii)] if $\bar z$ is a local minimizer of \eqref{eq:implicit_multipliers}, then
		due to \cref{thm:local_relationship}, explicit M-stationarity w.r.t.\ \emph{all} elements of
		$K(\bar z)$ is a reasonable necessary optimality condition for
		\eqref{eq:implicit_multipliers}	(under suitable constraint qualifications), and
	\item[(iii)] the constraint qualifications needed to infer implicit M-stationarity of
		a local minimizer $\bar z$ of \eqref{eq:implicit_multipliers}
		($H_M$ metrically subregular at some reference point) 
		and to obtain explicit M-stationarity from that,
		see \cref{prop:from_implicit_to_explicit_M_St,prop:from_implicit_to_explicit_M_stationarity_1}
		as well as \cref{cor:from_implicit_to_explicit_M_St},
		might still be weaker than the 
		constraint qualifications needed to infer explicit M-stationarity w.r.t.\
		\emph{all} $\bar\lambda\in K(\bar z)$ directly from 
		\eqref{eq:explicit_multipliers}.
\end{enumerate}
These observations underline that although the
use of implicit variables as explicit ones might be beneficial for computational purposes,
this transformation is disadvantageous since it comes for the price of additional artificial
local minimizers and potentially stronger constraint qualifications. 
Instead of treating implicit variables as explicit ones,
one should keep them implicit while exploiting the inherent underlying
problem structure as long as possible in order to infer reasonably weak constraint qualifications
ensuring validity of useful necessary optimality conditions at local minimizers.  

Let us mention some positive features of implicit variables regarding optimality
conditions. Clearly, whenever we are given some $\tilde z\in\R^n$ feasible to
\eqref{eq:implicit_multipliers} and some $\tilde\lambda\in K(\tilde z)$ such that
$\mathfrak H_M$ is metrically subregular at $((\tilde z,\tilde \lambda),((0,0),(0,0,0),0))$
while $\tilde z$ is \emph{not} explicitly M-stationary w.r.t.\ $\tilde \lambda$,
then $\tilde z$ cannot be a local minimizer of \eqref{eq:implicit_multipliers},
see \cref{prop:Stationarities_via_subreguarity}.
This means that implicit variables can be used to infer suboptimality conditions for
\eqref{eq:implicit_multipliers}.
Secondly, let us mention that it might happen that a practically \emph{useful}
representation of the coderivative of $H$ is not available for some 
underlying applications while the computation of the coderivative of $F$ and $G$ might
be possible in terms of initial problem data, see e.g.\ \cref{sec:app_multicriteria_optimization}. 
Then the explicit M-stationarity conditions
of \eqref{eq:implicit_multipliers} might be applicable while the implicit counterpart
is of limited practical use. Of course, under rather mild assumptions one can estimate the
coderivative of $H$ via the coderivatives of $F$ and $G$,
but this is not completely for free in general.

\subsection{Necessary optimality conditions under additional structural assumptions}
	\label{sec:optimality_conditions_via_special_structure}

In this section, we show how the situation from \cref{sec:M_Stationarity_conditions}
simplifies in specific settings that cover our example problems
from \cref{sec:examples}.
To this end, we impose that the following assumption on the problem data holds
throughout the section.
\begin{assumption}\label{ass:special_structure}
The set $\gph H$ is closed, and the problem data satisfies one of the following 
assumptions on the set-valued mapping $G$:
\begin{enumerate}
  	\item[(i)] $G(z,\lambda) := g(z,\lambda) - \Theta$ holds for all $z\in\R^n$ and $\lambda\in\R^m$
  		where $g\colon\R^n\times\R^m\to\R^s$ is a continuously differentiable function and 
  		$\Theta\subset\R^s$ is a closed set,
  	\item[(ii)] $G(z,\lambda):=\widetilde G(\lambda) + \tilde g(z)$	
  		holds for all $z\in\R^n$ and $\lambda\in\R^m$
  		where $\widetilde G\colon\R^m\tto\R^s$ is a set-valued mapping with a closed graph
  		and $\tilde g\colon\R^n\to\R^s$ is a continuously differentiable function.
\end{enumerate}
\end{assumption}
Let us point out that (i) covers the setting of bilevel programming from 
\eqref{eq:setting_bilevel_programming} as well as
the setting of cardinality-constrained optimization from
\eqref{eq:setting_CCMP}, while (ii) covers
the setting of evaluated multiobjective programming from 
\eqref{eq:setting_multiobjective_optimization}.

Statements (a) and (b) of \cref{lem:equality_coderivative_products} yield
that the inclusion \eqref{eq:coder_CalH} holds with equality under any
of the assumptions (i) or (ii).
Thus, due to \cref{rem:fuzzy_vs_explicit_M_stationarity}, 
fuzzy M-stationarity in fact coincides with explicit M-stationarity
and so we can work only with the latter.
Particularly, we will only consider the 
implicit and explicit M-stationarity conditions with the respective
constraint qualification being metric subregularity of $H_M$ and $\mathcal{H}_M$.

For the remaining part,
let us add the inner semicompactness of $\widehat K$ w.r.t.\ $\dom\widehat K$ at $(\bar z,0)$ 
for some feasible point $\bar z\in\R^n$ of \eqref{eq:implicit_multipliers} 
to the standing assumption. Then the comparison of the metric subregularity
conditions reduces to simply saying that
$H_M$ is metrically subregular at $(\bar z,(0,0))$ provided
$\mathcal{H}_M$ is metrically subregular at $((\bar z,\lambda),(0,0,0))$
for all $\lambda \in K(\bar z)$, 
see \cref{prop:MSR_of_H_via_cal_H}.

Finally, we discuss when implicit M-stationarity implies explicit M-stationarity.
Again, the statements (a) and (b) of \cref{lem:equality_coderivative_products} 
show that there is now no problem with the computation of
the coderivative of the mapping $\mathcal H$
and the same applies to the mapping $\widehat{\mathcal H}$
from \eqref{eq:simple_feasible_map} due to
the arguments above \cref{prop:from_implicit_to_explicit_M_stationarity_1}.
Particularly, we infer that $\widehat{\mathcal H}$ is metrically regular
at $((\bar z,0,\lambda),(0,0))$ for $\lambda\in K(\bar z)$ if and only
if the constraint qualification \eqref{eq:metric_regularity_CQ} is valid.
One can check that the latter is inherently satisfied in the setting (i)
since $G$ possesses the Aubin property in this case.
If (ii) holds, \eqref{eq:metric_regularity_CQ} boils down to
\begin{equation*}
    \ker D^*F(\bar z,\lambda) \cap D^* \widetilde G(\lambda, -\tilde g(\bar z))(0) = \{0\}.
\end{equation*}

We sum up all these arguments in the following proposition.
\begin{proposition}\label{prop:explicit_M_stationarity_in_particular_settings}
 	Let $\bar z\in\R^n$ be implicitly M-stationary for problem \eqref{eq:implicit_multipliers} 
 	and let $\widehat K$ be inner semicompact w.r.t.\ $\dom\widehat K$
 	at $(\bar z,0)$. Then the following assertions hold.
 	\begin{enumerate}
  		\item[(a)] Let $G$ be given as stated in (i).
  			Then $\bar z$ is explicitly M-stationary, i.e., there exist 
  			$\bar \lambda \in K(\bar z)$ and 
  			$\nu \in \mathcal N_{\Theta}(g(\bar z,\bar \lambda))$ such that
  			\[
  				0\in\partial f(\bar z)
            	+
            	D^*F(\bar z,\bar\lambda)(\nabla_{\lambda} g(\bar z,\bar \lambda)^\top \nu)
            	+ 
            	\nabla_{z} g(\bar z,\bar \lambda)^\top \nu
				+
				\mathcal N_M(\bar z).
			\]
 	\item[(b)] Let $G$ be given as stated in (ii) and let $\widehat{\mathcal H}$ be metrically subregular 
  			at $((\bar z,0,\lambda),(0,0))$
  			for each $\lambda \in K(\bar z)$. 
  			Then $\bar z$ is explicitly
  			M-stationary, i.e., there exist $\bar \lambda \in K(\bar z)$
  			and $\nu \in \R^s$ such that
  			\[
  				0\in\partial f(\bar z)
            	+
            	\bigl(
            		D^*F(\bar z,\bar\lambda)\circ D^* \widetilde G(\bar \lambda, -\tilde g(\bar z))
            	\bigr)(\nu)
				+ 
				\nabla \tilde g(\bar z)^\top \nu
				+
				\mathcal N_M(\bar z).
			\]
 	\end{enumerate}
\end{proposition}

The above proposition shows that in the setting (i)
from \cref{ass:special_structure},
one only has to impose metric subregularity of $H_M$ at some reference point $(\bar z,(0,0))$,
where $\bar z$ is a local minimizer of \eqref{eq:implicit_multipliers},
and some inner semicompactness of $\widehat K$ in order to come up
with fully explicit optimality conditions in terms of initial problem data. In
the setting (ii), an additional metric subregularity requirement on $\widehat{\mathcal H}$ is
necessary for that purpose.
As we will see, inner semicompactness of $\widehat K$ is inherent under reasonable assumptions
in the context of the problem settings from \cref{sec:examples}.
Thus, one might be tempted to say that in any of the settings from \cref{ass:special_structure},
an implicitly M-stationary point of \eqref{eq:implicit_multipliers} is likely to be
explicitly M-stationary w.r.t.\ at least one choice of the implicit variable. 

\subsection{Convexity and sufficient optimality conditions}\label{sec:sufficient_conditions}

We want to close our theoretical analysis of the abstract model \eqref{eq:implicit_multipliers}
with a brief look at sufficient optimality conditions in the presence of convexity.
More precisely, we investigate the problem of interest under the subsequently
stated standing assumption.
\begin{assumption}\label{ass:convexity}
	Let $f$ be a convex function, let $F$ and $G$ be convex set-valued mappings,
	and let $M$ be a convex set.
\end{assumption}

As we already mentioned in \cref{cor:convex_programs}, \cref{ass:convexity} ensures
that $H$ is convex as well. Observing that the graphs of $F$, $G$, and $H$ are
convex, their coderivatives are
given via the normal cone in the sense of convex analysis.
That is why we obtain the following result.
\begin{theorem}\label{thm:sufficient_optimality_conditions}\
	Let $\bar z\in\R^n$ be an implicitly, fuzzily, or explicitly M-stationary
	point of \eqref{eq:implicit_multipliers}.
	Then $\bar z$ is a minimizer of \eqref{eq:implicit_multipliers}.
\end{theorem}
\begin{proof}
	We only show the statement for explicitly M-stationary points.
	The remaining assertions can be derived similarly.\\
	Let $\bar z$ be explicitly M-stationary. Then we find $\bar\lambda\in K(\bar z)$,
	$\mu\in\R^m$, and $\nu\in\R^s$ as well as $\xi_1\in\partial f(\bar z)$,
	$\xi_2\in D^*F(\bar z,\bar\lambda)(\mu)$, $(\xi_3,\mu)\in D^*G((\bar z,\bar\lambda),0)(\nu)$,
	and $\xi_4\in\mathcal N_M(\bar z)$ with $0=\xi_1+\xi_2+\xi_3+\xi_4$.
	For each $(z,\lambda)\in\tilde Z$, we obtain
	\begin{align*}
		f(z)
		&
		\geq
		f(\bar z)+\xi_1^\top(z-\bar z)
		=
		f(\bar z)-\xi_2^\top(z-\bar z)-\xi_3^\top(z-\bar z)-\xi_4^\top(z-\bar z)\\
		&
		\geq
		f(\bar z)-\xi_2^\top(z-\bar z)-(-\mu)^\top(\lambda-\bar\lambda)
			-\xi_3^\top(z-\bar z)-\mu^\top(\lambda-\bar\lambda)-(-\nu)^\top(0-0)
			\\
		&
		\geq
		f(\bar z)
	\end{align*}
	by definition of the subdifferential and the normal cone in the sense of
	convex analysis. This shows that $(\bar z,\bar\lambda)$ is a
	global minimizer of \eqref{eq:explicit_multipliers}, i.e., $\bar z$ is a global
	minimizer of \eqref{eq:implicit_multipliers} due to \cref{thm:global_relationship}.
\end{proof}

\section{Consequences for certain problem classes}\label{sec:applications}

In this section, we discuss some of the results obtained in \cref{sec:lambda_as_variable}
by means of the example problems introduced in \cref{sec:examples}.

\subsection{Bilevel programming}\label{sec:app_bilevel_programming}

In this section, we take a look back at the bilevel programming problem  \eqref{eq:BPP_implicit_multipliers} from 
\cref{sec:bilevel_programming}. Throughout the section, the following additional 
standing assumption may hold.
\begin{assumption}\label{ass:lower_level_regularity_via_MSCQ}
The lower level feasibility mapping 
$y\tto g(y)-C$ is metrically subregular at all points $(y,0)$ 
belonging to its graph.
\end{assumption}

Clearly, the requirements from \cref{ass:lower_level_regularity_via_MSCQ} are inherent 
whenever a Robinson-type constraint qualification holds at
all lower level feasible points. In case $C:=\R^m_-$, this amounts to validity
of MFCQ at all points from $\Gamma$.
However, in many situations, the metric subregularity assumption 
from \cref{ass:lower_level_regularity_via_MSCQ} might be weaker.

Consulting \cref{sec:abstract_M_St} and \cite[Theorem~6.14]{RockafellarWets1998},
we find that the pre-image formula
\[
	\widehat{\mathcal N}_\Gamma(y)
	=
	\nabla g(y)^\top\widehat{\mathcal N}_C(g(y))
	=
	\nabla g(y)^\top\bigl[C^\circ\cap g(y)^\perp\bigr]
\]
holds for all $y\in\Gamma$.
Thus, the associated intermediate mapping $K\colon\R^{n_1}\times\R^{n_2}\tto\R^m$ is given by
\begin{align*}
	K(x,y)
	=
	\bigl\{
		\lambda\in C^\circ\,\big|\,
		\nabla_yj(x,y)+\nabla g(y)^\top\lambda=0,\,\lambda^\top g(y)=0
	\bigr\}
\end{align*}
which is nothing else but the so-called Lagrange multiplier mapping 
associated with \eqref{eq:lower_level}.
Keeping \cref{prop:inner_semicompactness_of_multiplier_map} and the continuity of
$\nabla_yj(\cdot,\cdot)$ in mind, $K$ is inner semicompact w.r.t.\ its domain
everywhere.
Thus, our \cref{thm:global_relationship,thm:local_relationship} recover the
results obtained in \cite{DempeDutta2012,DempeMefoMehlitz2018} under validity of
weaker constraint qualifications at the lower level. Indeed, we do not exploit the
local boundedness of the intermediate mapping $K$, which was the key idea in the
latter papers, but only its inner semicompactness. For that purpose, metric
subregularity of the feasibility mapping is enough.
Let us inspect the assertion of \cref{cor:local_relationship} in the light of the
present setting. Therefore, we fix a point $((\bar x,\bar y),\bar\lambda)\in\gph K$.
Relying on \cite[Lemma~4.44, Proposition~4.47]{BonnansShapiro2000} and observing that
$C$ is a cone, we obtain that
$K$ is inner semicontinuous at $((\bar x,\bar y),\bar\lambda)$ if the condition
\begin{equation}\label{eq:MFC}
	\nabla g(\bar y)^\top\lambda=0,\,\lambda^\top g(\bar y)=0,\,\lambda\in (C\cap\bar\lambda^\perp)^\circ
	\quad\Longrightarrow\quad
	\lambda=0
\end{equation}
is valid. Actually, this already implies $K(\bar x,\bar y)=\{\bar\lambda\}$ and that $K$
is so-called locally upper Lipschitz continuous at $(\bar x,\bar y)$, 
see \cite[Section~2.3]{BonnansShapiro2000} for a definition.
In general, \eqref{eq:MFC} is more restrictive than validity of GMFCQ at $\bar y$.
Noting that \eqref{eq:MFC} depends on the multiplier $\bar\lambda$ and, thus, implicitly
on the objective function $j$, this condition is not a constraint qualification in the 
narrower sense. In the setting of standard nonlinear programming, i.e., where
$C:=\R_-^m$ holds, \eqref{eq:MFC} is referred to as the 
\emph{strict Mangasarian--Fromovitz condition} which is known to be the weakest condition
implying uniqueness of Lagrange multipliers, see \cite{Wachsmuth2013}.
In the more general setting where $C$ is a polyhedral, convex cone, 
\eqref{eq:MFC} is implied by the so-called
\emph{non-degeneracy} condition
\begin{align*}
	\nabla g(\bar y)^\top\lambda=0,\,
	\lambda\in\bigl(C^\circ\cap g(\bar y)^\perp-C^\circ\cap g(\bar y)^\perp\bigr)
	\quad\Longrightarrow\quad
	\lambda=0
\end{align*}
which corresponds to the prominent \emph{linear independence constraint qualification} (LICQ)
in standard nonlinear programming. 
It can be distilled from \cite[Theorem~3.2]{DempeDutta2012}
that a bilevel optimization problem and its reformulated single-level counterpart, where lower level
multipliers are treated as explicit variables, are equivalent w.r.t.\ local minimizers 
whenever LICQ is valid at each point which is feasible for the lower level problem.
In this regard, the above observations provide another generalization of the results from
\cite{DempeDutta2012}.
Let us also mention that if $C$ is polyhedral,
we can actually obtain stronger properties of $K$ which correspond to inner semicontinuity and inner semicompactness but with linear rate of change (called inner calmness and inner calmness*, respectively, in \cite{Benko2019,BenkoMehlitz2020}, where more details can be found).

Recalling the setting from \eqref{eq:setting_bilevel_programming}, we have
$H(z):=\nabla_yj(z)+\widehat{\mathcal N}_\Gamma(y)$ for all $z:=(x,y)$.
In this regard, the implicit M-stationarity conditions of \eqref{eq:BPP_implicit_multipliers}
at some feasible point $\bar z:=(\bar x,\bar y)$
reduce to the existence of $\nu\in\R^{n_2}$ such that
\[
	(0,0)\in\partial f(\bar z)
		+\nabla^2_{yz}j(\bar z)^\top\nu
		+ \{0\}\times D^*\widehat{\mathcal N}_\Gamma(\bar y,-\nabla_yj(\bar z))(\nu)
		+\mathcal N_S(\bar x)\times\{0\}.
\]
Some recent progress in the field of variational analysis even allows to
calculate or estimate the appearing coderivative of the normal cone mapping 
$\widehat{\mathcal N}_\Gamma\colon\R^{n_2}\tto\R^{n_2}$
in terms of initial data under suitable assumptions, see e.g.\
\cite{GfrererOutrata2016,GfrererOutrata2016c,GfrOut17}.
On the other hand, the explicit M-stationarity conditions of \eqref{eq:BPP_implicit_multipliers}
reduce to the existence of $\bar\lambda\in K(\bar z)$ and $\nu\in\R^{n_2}$ which satisfy
\begin{align*}
	(0,0)&\in
		\partial f(\bar z)
		+\nabla^2_{yz}j(\bar z)^\top\nu
		+\bigl\{\bigl(
			0,\mathsmaller\sum\nolimits_{i=1}^m\bar\lambda_i\nabla^2g_i(\bar y)\nu
			\bigr)\bigr\}\\
		&\qquad
		+\{0\}\times 
			D^*(\widehat{\mathcal N}_C\circ g)(\bar y,\bar\lambda)(\nabla g(\bar y)\nu)
		+\mathcal N_S(\bar x)\times\{0\},
\end{align*}
see statement (a) of \cref{prop:explicit_M_stationarity_in_particular_settings}.
Above, $g_1,\ldots,g_m\colon\R^{n_2}\to\R$ denote the component functions of $g$.
Naturally, this stationarity system might be deduced directly from the above system of
implicit M-stationarity with the aid of the chain rule from \cref{lem:chain_rule}.
Consulting \cref{prop:inner_semicompactness_of_multiplier_map} once more, however, we
see that the intermediate mapping $\widehat K$ from \eqref{eq:simple_intermediate_map},
which takes the form
\[
	\widehat{K}((x,y),w)
	=
	\bigl\{
		\lambda\in C^\circ\,\big|\,
		-w+\nabla_yj(x,y)+\nabla g(y)^\top\lambda=0,\,\lambda^\top g(y)=0
	\bigr\},
\]
is also inner semicompact w.r.t.\ its domain at all points $((\bar x,\bar y),0)$
that satisfy the relation $(\bar x,\bar y)\in\dom K$. In this regard, statement (a) of
\cref{prop:explicit_M_stationarity_in_particular_settings} guarantees that each
implicitly M-stationary point of \eqref{eq:BPP_implicit_multipliers} is automatically
explicitly M-stationary as well.
This observation generalizes the classical results from \cite{MordukhovichOutrata2007}.

Note that the coderivative of $\widehat{\mathcal N}_C\circ g$ might be computed or estimated using
again a suitable chain rule while observing that the associated intermediate mapping from
\eqref{eq:intermediate_mapping} is naturally inner semicontinuous w.r.t.\ its domain 
at each point of its graph by continuity of $g$. 
This way, one might be tempted to call the stationarity system
\begin{align*}
	(0,0)&\in
		\partial f(\bar z)
		+\nabla^2_{yz}j(\bar z)^\top\nu
		+\bigl\{\bigl(
			0,\mathsmaller\sum\nolimits_{i=1}^m\bar\lambda_i\nabla^2g_i(\bar y)\nu
			\bigr)\bigr\}\\
		&\qquad
		+\{0\}\times \nabla g(\bar y)^\top 
			D^*\widehat{\mathcal N}_C(g(\bar y),\bar\lambda)(\nabla g(\bar y)\nu)
		+\mathcal N_S(\bar x)\times\{0\}
\end{align*}
the \emph{fully} explicit M-stationarity system of \eqref{eq:BPP_implicit_multipliers}
(w.r.t.\ $\bar\lambda$).
We want to point out that this system can be written equivalently as
\begin{align*}
	&(0,0)\in
		\partial f(\bar z)
		+\nabla^2_{yz}j(\bar z)^\top\nu
		+\bigl\{\bigl(
			0,\mathsmaller\sum\nolimits_{i=1}^m\bar\lambda_i\nabla^2g_i(\bar y)\nu
				+\nabla g(\bar y)^\top \mu
			\bigr)\bigr\}
		+\mathcal N_S(\bar x)\times\{0\}\\
	&(\mu,-\nabla g(\bar y)\nu)\in\mathcal N_{\gph{\widehat{\mathcal N}_C}}(g(\bar y),\bar\lambda),
\end{align*}
corresponding to the M-stationarity system of the optimization problem 
\begin{equation}\label{eq:MPCC_reformulation_of_BPP}
	\begin{aligned}
	f(x,y)&\,\to\,\min\limits_{x,y,\lambda}\\
	x&\,\in\,S\\
	\nabla_yj(x,y)+\nabla g(y)^\top\lambda&\,=\,0\\
	(g(y),\lambda)&\,\in\,\gph\widehat{\mathcal N}_C.
	\end{aligned}
\end{equation}
Observing that $C$ is a convex cone, we have
\[
	\gph\widehat{\mathcal N}_C
	=
	\bigl\{(a,b)\in C\times C^\circ\,\big|\,a^\top b=0\bigr\},
\]
and this motivates us to call \eqref{eq:MPCC_reformulation_of_BPP} a generalized
\emph{mathematical problem with complementarity constraints}, see \cite[Section~4]{FrankeMehlitzPilecka2018} or \cite{Wachsmuth2015}. Following the above arguments and
keeping \cref{sec:abstract_M_St} in mind, we obtain the following result.
\begin{corollary}\label{cor:fully_explicit_M_St_for_BPP}
	Let $\bar z:=(\bar x,\bar y)\in\R^{n_1}\times\R^{n_2}$ be a local minimizer
	of \eqref{eq:BPP_implicit_multipliers} and assume that, for some $\bar\lambda\in K(\bar z)$,
	the mapping
	\[
		(x,y,\lambda)
		\tto
		\left(	x-S,
				\nabla _yj(x,y)+\nabla g(y)^\top\lambda,
				(g(y),\lambda)-\gph\widehat{\mathcal N}_C
		\right)
	\]
	is metrically subregular at $((\bar x,\bar y,\bar\lambda),(0,0,(0,0)))$.
	Then $\bar z$ is fully explicitly M-stationary for \eqref{eq:BPP_implicit_multipliers}
	w.r.t.\ $\bar\lambda$.
\end{corollary}

Using the above terminology, we would like to mention that the authors in
\cite{AdamHenrionOutrata2018} discussed the implicit and fully explicit M-stationarity
conditions of \eqref{eq:BPP_implicit_multipliers} in case $C:=\R^m_-$. They came up
with refined conditions ensuring that the feasibility mapping from \cref{cor:fully_explicit_M_St_for_BPP}
is indeed metrically subregular at some reference point, see \cite[Theorem~8]{AdamHenrionOutrata2018}.
From the viewpoint of applicability, the system of fully explicit M-stationarity might be 
the most useful one among the stated ones since the appearing coderivative 
of the normal cone mapping $\widehat{\mathcal N}_C$
is computable in some situations where the set $C$ is \emph{simple},
see e.g.\ the proof of \cite[Theorem 2]{DoRo96}
and \cite[Theorem 2.12]{GfrererOutrata2016b}.
Let us mention that yet another M-stationarity-type system associated with 
\eqref{eq:BPP_implicit_multipliers} has been derived in the recent paper \cite{GfrererYe2019}.

\subsection{Evaluating weakly efficient points in multicriteria optimization}
	\label{sec:app_multicriteria_optimization}
	
Let us consider \eqref{eq:EMOP_implicit_multipliers} under the assumptions
from \cref{sec:vector_optimization}. Particularly, we consider problem
\eqref{eq:implicit_multipliers} with the setting from 
\eqref{eq:setting_multiobjective_optimization}.
The associated intermediate mapping $K\colon\R^n\tto\R^m$ reads as
\[
	\forall z\in\R^n\colon\quad
	K(z)=\{\lambda\in \Delta\,|\,z\in\Psi(\lambda)\}.
\]
The subsequent lemma provides essential foundations of our analysis.
\begin{lemma}\label{lem:isc_EMOP}
	The mapping $G$ from \eqref{eq:setting_multiobjective_optimization}
	possesses a closed graph while $K$ from above is inner semicompact w.r.t.\ its domain
	everywhere.
\end{lemma}
\begin{proof}
	First, we show that $G$ possesses a closed graph.
	Choose sequences $\{z_k\}_{k\in\N}\subset\R^n$, $\{\lambda_k\}_{k\in\N}\subset\R^m$,
	and $\{w_k\}_{k\in\N}\subset\R^n$ such that $z_k\to z$, $\lambda_k\to\lambda$, and
	$w_k\to w$ for some $z\in\R^n$, $\lambda\in\R^m$, and $w\in\R^n$ are valid 
	while $w_k\in G(z_k,\lambda_k)$ holds for all $k\in\N$. 
	Particularly, we find $\{\lambda_k\}_{k\in\N}\subset\Delta$ and
	$z_k+w_k\in\Psi(\lambda_k)$, i.e.,
	\[
		\forall k\in\N,\,\forall z'\in\Gamma\colon\quad
			\lambda_k^\top j(z_k+w_k)\leq \lambda_k^\top j(z')
	\]
	and $z_k+w_k\in\Gamma$ for all $k\in\N$. Taking the limit $k\to\infty$ and observing
	that $j$ is continuous while $\Gamma$ and $\Delta$ are closed, we infer $z+w\in\Gamma$ and
	$\lambda^\top j(z+w)\leq \lambda^\top j(z')$ for all $z'\in\Gamma$.
	This yields $z+w\in\Psi(\lambda)$ which equals $w\in G(z,\lambda)$, i.e.,
	$((z,\lambda),w)\in\gph G$. Particularly, $\gph G$ is closed.
	
	By definition, $F$ is locally bounded everywhere which is why the lemma's
	assertion regarding $K$ follows directly from \cref{lem:inner_semicompactness_of_K}.
\end{proof}

Now, \cref{thm:global_relationship,thm:local_relationship} can be used to 
infer the precise relationship between \eqref{eq:EMOP_implicit_multipliers}
and its associated counterpart \eqref{eq:explicit_multipliers} w.r.t.\
global and local minimizers, respectively.
Related results can be obtained in the slightly more general
context of \emph{semivectorial} bilevel programming where the lower level
decision maker has to solve a multiobjective optimization problem, see \cite{DempeMehlitz2018}.

Let us now focus on optimality conditions for \eqref{eq:EMOP_implicit_multipliers}.
Clearly, the associated implicit M-stationarity conditions take the form
\[
	0\in\partial f(\bar z)+\mathcal N_{\Gamma_{\textup{we}}}(\bar z)
\]
for an arbitrary feasible point $\bar z\in\R^n$. However, these conditions are
of limited practical use due to an essential lack of knowledge regarding the 
variational geometry of $\Gamma_\textup{we}$.
On the other hand,
invoking \cref{prop:explicit_M_stationarity_in_particular_settings},
we find that the explicit M-stationarity conditions of \eqref{eq:EMOP_implicit_multipliers}
take the following form: there exist $\bar\lambda\in K(\bar z)$ and $\nu\in\R^n$ such that
\[
	\nu\in\partial f(\bar z),
	\qquad
	0\in D^*\Psi(\bar\lambda,\bar z)(\nu)+\mathcal N_\Delta(\bar\lambda).
\]
Taking into account that 
\[
	\forall\lambda\in\Delta\colon\quad
	\Psi(\lambda)
	=
	\left\{
		z\in\R^n\,\middle|\,0\in\nabla j(z)^\top\lambda+\widehat{\mathcal N}_\Gamma(z)
	\right\}
\]
holds, we refer to \cite{DoRo14,GfrererOutrata2016b} for a broader
view on solution mappings of parametrized variational systems
and estimates of their generalized derivatives.
Particularly, in case where $j$ is twice continuously differentiable, we obtain the additional estimate
\begin{equation*}
	D^*\Psi(\bar\lambda,\bar z)(\nu)
	\subset
	\left\{
		\nabla j(\bar z)\mu\,\middle|\,
		-\nu\in\mathsmaller\sum\nolimits_{i=1}^m\bar\lambda_i\nabla^2j_i(\bar z)\mu 
		+
		D^*\widehat{\mathcal N}_{\Gamma}\bigl(\bar z,-\nabla j(\bar z)^\top\bar\lambda\bigr)(\mu)
	\right\}
	+
	\mathcal N_\Delta(\bar\lambda)
\end{equation*}
applying the pre-image rule to 
\[
	\gph \Psi 
	= 
	\left\{
	(\lambda,z) \,\middle|\, (z,-\nabla j(z)^\top\lambda,\lambda) 
	\in 
	\gph \widehat{\mathcal N}_\Gamma\times\Delta
	\right\}
\]
which is possible whenever the feasibility mapping 
$(\lambda,z)\tto(z,-\nabla j(z)^\top\lambda,\lambda)-\gph\widehat{\mathcal N}_\Gamma\times\Delta$
is metrically subregular at $((\bar\lambda,\bar z),(0,0,0))$,
see \cref{sec:abstract_M_St}.
The coderivative of the normal cone mapping associated with $\Gamma$ can be estimated
from above in several interesting situations covering the setting where $\Gamma$ is
the pre-image of a closed and convex (polyhedral) set under a sufficiently smooth mapping, see the references mentioned in \cref{sec:app_bilevel_programming}.
Using a related approach, fully explicit necessary optimality condition for 
semivectorial bilevel optimization problems are derived in \cite{Zemkoho2016}.

Observing that $F$ and $G$ are metrically subregular mappings in the present setting
\eqref{eq:setting_multiobjective_optimization}, 
metric subregularity of the mapping
$\widetilde\Sigma\colon\R^n\times\R^m\tto\R^m\times\R^{m+n}$ given by
\[
	\forall z\in\R^n,\,\forall\lambda\in\R^m\colon\quad
	\widetilde\Sigma(z,\lambda):=\bigl(\lambda-\Delta\bigr)\times\bigl((\lambda,z)-\gph\Psi\bigr)
\]
at $((\bar z,\bar\lambda),(0,(0,0)))$ for some $\bar\lambda\in K(\bar z)$ is already
enough to obtain that the associated mapping $\mathcal H_M$ is metrically subregular
at $((\bar z,\bar\lambda),(0,0))$, see \cref{lem:metric_subregularity_of_productsII}.
Note that $\widetilde\Sigma$ results from the mapping $\Sigma$ given in 
\cref{lem:metric_subregularity_of_productsII} by deleting the trivial factor $\R^n$ which
does not influence the metric subregularity property.
In the presence of this metric subregularity, $\bar z$ is an explicitly M-stationary point of
\eqref{eq:EMOP_implicit_multipliers} provided $\bar z$ is a local minimizer of this
program, see \cref{prop:Stationarities_via_subreguarity} and the arguments in 
\cref{sec:optimality_conditions_via_special_structure}.

\subsection{Cardinality-constrained programming}\label{sec:app_CCMP}

Finally, we investigate the setting of cardinality-constrained optimization from
\cref{sec:CCMP}.
In the context of \eqref{eq:setting_CCMP}, the mapping $K\colon\R^n\rightrightarrows\R^n$
is given by
\[
	\forall z\in\R^n\colon\quad
	K(z)=
	\left\{\lambda\in[0,1]^n\,\middle|\,
			\mathsmaller\sum\nolimits_{i\in I^0(z)}\lambda_i\geq n-\kappa
			,\,\forall i\in I^{\pm}(z)\colon\,\lambda_i=0
	\right\}
\]
where we used
\[
	I^{\pm}(z):=\{i\in\{1,\ldots,n\}\,|\,z_i\neq 0\},
	\qquad
	I^0(z):=\{i\in\{1,\ldots,n\}\,|\,z_i=0\}
\]
for arbitrary $z\in\R^n$. 
Obviously, $K$ is locally bounded and, thus, inner semicompact
w.r.t.\ the set $\dom K=\{z\in\R^n\,|\,\norm{z}{0}\leq\kappa\}$ at
all points of its domain.
The images of $K$ are polytopes (i.e., bounded polyhedrons). 
Furthermore, $K(z)$ is a singleton if and only if
$z$ satisfies $\norm{z}{0}=\kappa$. In this regard, \cref{thm:global_relationship}
recovers \cite[Theorem~3.2]{BurdakovKanzowSchwartz2016}.
On the other hand, the results \cite[Proposition~3.5, Theorem~3.6]{BurdakovKanzowSchwartz2016}
are consequences of \cref{thm:local_relationship}. However, our result even applies to
settings where $\norm{z}{0}<\kappa$ is valid, and, thus, clarifies the situation
in \cite[Examples~1 and 2]{BurdakovKanzowSchwartz2016}.

Let us mention first that due to the arguments from \cref{sec:CCMP}, the mapping $H$ takes
the form $z\tto D_\kappa-z$ in the situation at hand, 
see \eqref{eq:disjunctive_reformulation_cardinality_constraint},
and we note that it is polyhedral by nature of $D_\kappa$
(since $D_\kappa$ is the union of finitely many convex, polyhedral sets).
The implicit M-stationarity conditions of \eqref{eq:CCMP} w.r.t.\ 
a feasible point $\bar z\in\R^n$ are given by
\[
	0\in\partial f(\bar z)+\mathcal N_{D_\kappa}(\bar z)+\mathcal N_M(\bar z).
\]
By means of the formula
\[
	\mathcal N_{D_\kappa}(\bar z)
	=
	\{\nu\in\R^n\,|\,\norm{\nu}{0}\leq n-\kappa,\,\forall i\in I^{\pm}(\bar z)\colon\,\nu_i=0\},
\]
which can be distilled from \cite[Lemma~2.3]{PanXiuFan2017}, we obtain a
reasonable optimality condition as soon as the variational structure of
$M$ is nice enough. Let us point out that whenever $M$ is the union of finitely
many convex polyhedral sets, then each local minimizer of the associated problem 
\eqref{eq:CCMP} is implicitly M-stationary by \cref{rem:handling_of_abstract_constraints} 
since the feasibility mapping 
$z\tto (z-D_\kappa)\times(z-M)$ is polyhedral in this situation and, thus,
metrically subregular at all points of its graph. In case where the variational
geometry of $M$ is more difficult, metric subregularity of this feasibility mapping
at $(\bar z,(0,0))$ is still enough to infer implicit M-stationarity whenever $\bar z$
is a local minimizer of \eqref{eq:CCMP}. In case where $M:=\{z\in\R^n\,|\,\Phi(z)\in\Omega\}$
holds for a continuously differentiable mapping $\Phi\colon\R^n\to\R^p$ and a closed set
$\Omega\subset\R^p$, the associated Mordukhovich criterion, which ensures metric
regularity of the feasibility mapping and validity of the pre-image rule for
the estimation of the limiting normal cone to $D_\kappa\cap M$ at $\bar z$, takes the form
\[
\bigl\Vert\nabla\Phi(\bar z)^\top\eta\bigr\Vert_0\leq n-\kappa,\,
			\eta\in\mathcal N_\Omega(\Phi(\bar z)),\,
			\forall i\in I^{\pm}(\bar z)\colon\,
				\bigl(\nabla\Phi(\bar z)^\top\eta\bigr)_i=0
		\quad
	\Longrightarrow
	\quad
	\eta=0.
\]
As we pointed out in \cref{sec:optimality_conditions_via_special_structure}, the fuzzy
and explicit M-stationarity conditions of \eqref{eq:CCMP} which follow from the
setting \eqref{eq:setting_CCMP} coincide.
Performing some calculations and keeping statement (a) of \cref{prop:explicit_M_stationarity_in_particular_settings} in mind, these explicit M-stationarity
conditions take the following form at a reference point
$\bar z$: 
there exists some $\bar\lambda\in K(\bar z)$ such that
\[
	0\in\partial f(\bar z)
				+\{\nu\in\R^n\,|\,\forall i\in I^{\pm}(\bar z)\colon\,\nu_i=0\}
				+\mathcal N_M(\bar z).
\]
We note that this condition does not depend on $\bar\lambda$ at all.
It corresponds to the M-stationarity notion discussed in 
\cite{BucherSchwartz2018,BurdakovKanzowSchwartz2016,CervinkaKanzowSchwartz2016}.
At the first glance, we observe that the implicit M-stationarity condition is more restrictive
than its explicit counterpart. In light of statement (a) of 
\cref{prop:explicit_M_stationarity_in_particular_settings}, this is not surprising since the 
intermediate mapping $\widehat K$ associated with the present setting is inner semicompact w.r.t.\
its domain at each point of the latter. 
Recall that the explicit M-stationarity system does not depend on
the precise choice of the implicit variable from $K(\bar z)$, i.e., implicit M-stationarity
of $\bar z$ implies explicit M-stationarity w.r.t.\ all implicit variables from $K(\bar z)$. 
On the other hand, let us mention that in case of $\bar z:=0$ being feasible to 
\eqref{eq:CCMP}, it is
always an explicitly M-stationary point of \eqref{eq:CCMP} w.r.t.\ each implicit variable 
from $K(\bar z)$ while this point is not necessarily implicitly M-stationary.

\section{Conclusions}\label{sec:conclusions}

The essential message of this paper says that whenever optimization problems with
implicit variables are under consideration, it is better to leave them implicit
as long as possible. In many practically relevant situations like
cardinality-constrained programming, see \cref{sec:app_CCMP}, this procedure leads to more
restrictive necessary optimality conditions which hold under less restrictive
constraint qualifications (in a certain sense). Furthermore, the implicit formulation avoids the 
appearance of artificial local minimizers, and this might be beneficial not only
from a theoretical but also from a numerical point of view. 
It is, thus, always desirable to explore the inherent problem structure of the original
problem instead of making its implicit variables explicit for a non-negligible price.
Exemplary, let us mention that a convincing variational description of the
weakly efficient set of a multiobjective optimization problem which avoids the
use of scalarization variables is likely to enhance the results from 
\cref{sec:app_multicriteria_optimization}.



\end{document}